\documentclass[12pt]{article}
\usepackage{amsmath,amssymb,amsthm,amsfonts,amstext,amsbsy,amscd,mathrsfs,graphics}

\newtheorem{defin}{Definition}
\newtheorem{lemma}{Lemma}
\newtheorem{thm}{Theorem}
\newtheorem{rem}{Remark}
\newtheorem{rems}{Remarks}
\newtheorem{prop}{Proposition}
\newtheorem{corl}{Corollary}

\usepackage{amsmath}
\usepackage{amssymb}
\usepackage{graphicx}

\usepackage{color}

\def\Z{\mathbb Z}

\def\L{\mathbb L}
\newcommand{\N}{\mathbb N}
\newcommand{\R}{\mathbb R}
\newcommand{\Q}{\mathbb Q}

\begin{document}
\title{{{{Stochastic flows related to Walsh Brownian motion}}}}
\author{Hatem Hajri\\ D\'epartement de Math\'ematiques\\ Universit\'e Paris-Sud 11, 91405 Orsay, France\\Hatem.Hajri@math.u-psud.fr}
\date{}

\maketitle

\begin{abstract}
We define an equation on a simple graph which is an extension of Tanaka's equation and the skew Brownian motion equation. We then apply the theory of transition kernels developed by Le Jan and Raimond and show that all the solutions can be classified by probability measures. 
\end{abstract}
\textbf{Key words}: Stochastic flows of kernels, Skew Brownian motion, Walsh Brownian motion.\\
\textbf{AMS 2000 Subject Classification}: Primary 60H25; Secondary 60J60.\\
Submitted to EJP on January 17, 2011, final version accepted July 18, 2011.
%*****************************************************************************

\section{Introduction and main results.}
In \cite{MR1905858}, \cite{MR2060298} Le Jan and Raimond have extended the classical theory of stochastic flows to include flows of probability kernels.  Using the Wiener chaos decomposition, it was shown that non Lipschitzian stochastic differential equations have a unique Wiener measurable solution given by random kernels. Later, the theory was applied in \cite{MR2235172} to the study of Tanaka's equation:
\begin{equation}\label{ferj}
\varphi_{s,t}(x)=x+\int_{s}^t \textrm{sgn}(\varphi_{s,u}(x))W(du),\ \ s\leq t, x\in\R,
\end{equation}
where $\textrm{sgn}(x)=1_{\{x>0\}}-1_{\{x\leq0\}}, W_t=W_{0,t} 1_{\{t>0\}}-W_{t,0} 1_{\{t\leq0\}}$ and $(W_{s,t}, s\leq t)$ is a real white noise (see Definition 1.10 \cite{MR2060298}) on a probability space $(\Omega,\mathcal A,\mathbb P)$. If $K$ is a stochastic flow of kernels (see Definition \ref{ghazel} below) and $W$ is a real white noise, then by definition, $(K,W)$ is a solution of Tanaka's SDE if for all $s\leq t, x\in\R$, $f\in C^2_b(\R)$ ($f$ is $C^2$ on $\R$ and $f',f''$ are bounded) 
\begin{equation}\label{dom}
K_{s,t}f(x)=f(x)+\int_s^tK_{s,u}(f'\textrm{sgn})(x)W(du)+\frac{1}{2}\int_s^t K_{s,u}f''(x)du\ \ a.s.
\end{equation}
%When $K=\delta_{\varphi}$ is a Dirac flow (flow of mappings), then $K$ solves (\ref{dom}) if and only if $\varphi$ solves (\ref{ferj}).
It has been proved \cite{MR2235172} that each solution flow of (\ref{dom}) can be characterized by a probability measure on $[0,1]$ which entirely determines  its law. Define $$\tau_{s}(x)=\inf\{r\geq s : W_{s,r}=-|x|\}, s, x\in\R.$$
Then, the unique $\mathcal F^W$ adapted solution (Wiener flow) of (\ref{dom}) is given by
\[
K^{W}_{s,t}(x)=\delta_{x+\textrm{sgn}(x)W_{s,t}} 1_{\{t\leq \tau_{s}(x)\}}+\frac{1}{2}(\delta_{W^+_{s,t}}+\delta_{-W^+_{s,t}}) 1_{\{t>\tau_{s}(x)\}},\ W_{s,t}^+:=W_{s,t}-\displaystyle\inf_{u\in[s,t]}W_{s,u}.
\]
Among solutions of (\ref{dom}), there is only one flow of mappings (see Definition \ref{york} below) which has been already studied in \cite{MR1816931}.
\\
We now fix $\alpha\in]0,1[$ and consider the following SDE having a less obvious extension to kernels:
\begin{equation}\label{10}
X_{t}^{s,x}=x+W_{s,t}+(2\alpha-1)\tilde{L}_{s,t}^{x},\ \ t\geq s, x\in\R,
\end{equation}
where
$$\tilde{L}_{s,t}^{x}=\lim_{\varepsilon \rightarrow 0^+} \frac{1}{2\varepsilon}\int_{s}^{t} 1_{\{|X_{u}^{s,x}|\leq \varepsilon\}}du\ \ \ (\textrm{The symmetric local time}).$$ 
Equation (\ref{10}) was introduced in \cite{MR606993}. For a fixed initial condition, it has a pathwise unique solution which is distributed as the skew Brownian motion $(SBM)$ with parameter $\alpha$ $(SBM(\alpha))$. It was shown in \cite{MR1837288} that when $\alpha\neq\frac{1}{2}$, flows associated to (\ref{10}) are coalescing and a deeper study of (\ref{10}) was provided later in \cite{MR1880238} and \cite{MR2094439}. Now, consider the following generalization of (\ref{ferj}):
\begin{equation}\label{frouja}
X_{s,t}(x)=x+\int_{s}^t \textrm{sgn}(X_{s,u}(x))W(du)+(2\alpha-1)\tilde{L}_{s,t}^{x}(X),\ \ s\leq t, x\in\R,
\end{equation}
where
$$
\tilde{L}_{s,t}^{x}(X)=\displaystyle{\lim_{\varepsilon \rightarrow 0^+ }}\frac{1}{2\varepsilon }\int_s^t 1_{\{|X_{s,u}(x)|\leq \varepsilon\}}du. $$
Each solution of (\ref{frouja}) is distributed as the $SBM(\alpha)$. By Tanaka's formula for symmetric local time (\cite{MR1725357} page 234),
$$|X_{s,t}(x)|=|x|+\int_{s}^t \widetilde{\textrm{sgn}}(X_{s,u}(x))dX_{s,u}(x)+\tilde{L}_{s,t}^{x}(X),$$
where $\widetilde{\textrm{sgn}}(x)=1_{\{x>0\}}-1_{\{x<0\}}$. By combining the last identity with (\ref{frouja}), we have
\begin{equation}\label{tissma3}
|X_{s,t}(x)|=|x|+W_{s,t}+\tilde{L}_{s,t}^{x}(X).\
\end{equation}
The uniqueness of solutions of the Skorokhod equation (\cite{MR1725357} page 239) entails that
\begin{equation}\label{3alim}
|X_{s,t}(x)|=|x|+W_{s,t}-\min_{s\leq u\leq t}[(|x|+W_{s,u})\wedge0].
\end{equation}
Clearly (\ref{tissma3}) and (\ref{3alim}) imply that $\sigma(|X_{s,u}(x)|; s\leq u\leq t)=\sigma(W_{s,u}; s\leq u\leq t)$ which is strictly smaller than $\sigma(X_{s,u}(x); s\leq u\leq t)$ and so $X_{s,\cdot}(x)$ cannot be a strong solution of (\ref{frouja}). For these reasons, we call (\ref{frouja}) Tanaka's SDE related to $SBM(\alpha)$.
\\
From now on, for any metric space $E$, $C(E)$ (respectively $C_b(E)$) will denote the space of all continuous (respectively bounded continuous) $\R$-valued functions on $E$. Let
\begin{itemize}
 \item $C_b^2(\R^*)=\{f \in C(\R) : f\ \textrm{is twice derivable on}\ \R^*, f', f'' \in C_b({\R}^{*}),\  f'_{|]0,+\infty[}, f''_{|]0,+\infty[}\\ (\textrm{resp.}\ f'_{|]-\infty,0[} ,f''_{|]-\infty,0[})\ \textrm{have right}\ (\textrm{resp. left})\ \textrm{limit in}\ 0\}.$
\item $D_{\alpha}=\{f\in C_b^2(\R^*) : \alpha f'(0+)=(1-\alpha) f'(0-)\}$.
\end{itemize}
For $f\in D_{\alpha}$, we set by convention $f'(0)=f'(0-), f''(0)=f''(0-)$. By It\^o-Tanaka's formula (\cite{MR2280299} page 432) or Freidlin-Sheu formula (see Lemma 2.3 \cite{MR1743769} or Theorem \ref{palestine} in Section 2) and Proposition \ref{vi} below, both extensions to kernels of (\ref{10}) and (\ref{frouja}) may be defined by 
\begin{equation}\label{laurent}
K_{s,t}f(x)=f(x)+\int_s^tK_{s,u}(\varepsilon f')(x)W(du)+\frac{1}{2}\int_s^t K_{s,u}f''(x)du,\ \ f\in D_{\alpha},
\end{equation}
where $\varepsilon(x)=1$ (respectively $\varepsilon(x)=\textrm{sgn}(x)$) in the first (respectively second) case, but due to the pathwise uniqueness of (\ref{10}), the unique solution of (\ref{laurent}) when $\varepsilon(x)=1$, is $K_{s,t}(x)=\delta_{X_{t}^{s,x}}$ (this can be justified by the weak domination relation, see (\ref{h})). Our aim now is to define an extention of (\ref{laurent}) related to Walsh Brownian motion in general. The latter process was introduced in \cite{MR509476} and will be recalled in the coming section. We begin by defining our graph.

\begin{defin}(Graph G)\\
Fix $N\geq 1$ and $\alpha_1,\cdots,\alpha_N > 0$ such that $\displaystyle{\sum_ {i=1}^{N}\alpha_i=1}$.\\
In the sequel $G$ will denote the graph below (Figure 1) consisting of $N$ half lines $(D_i)_{1\leq i\leq N}$ emanating from $0$. Let $\vec{e}_i$ be a vector of modulus $1$ such that $D_i=\{h\vec{e}_i,h\geqslant 0\}$ and define for all function $f : G\longrightarrow{\R}$ and $i\in[1,N]$, the mappings :
$$\begin{array}{ccccc}
f_i & : & {\R_+} & \longrightarrow & {\R}  \\
& &h& \longmapsto & f(h\vec{e}_i) \\
\end{array}$$
Define the following distance on $G$:
$$\begin{array}{ll}
d(h\vec{e}_i,h'\vec{e}_j)=\begin{cases}
h+h'&\text{if}\  i\neq j, (h,h')\in \R_+^2,\\
|h-h'|\ &\text {if}\ i=j, (h,h')\in \R_+^2.\\
\end{cases}
\end{array}$$
For $x\in G$, we will use the simplified notation $|x|:=d(x,0)$.\\
We equip $G$ with its Borel $\sigma$-field $\mathcal{B}(G)$ and use the notation $G^*=G\setminus\{0\}$. Now, define\\
$\bullet$ $C_b^2(G^*)=\{f \in C(G) :
\forall i\in [1,N], f_i\ \textrm{is twice derivable on}\ \R^*_+, f'_i, f''_i \in C_b({\R}_{+}^{*})\ \textrm{and both}\\
\ \textrm{have finite limits at}\ 0+\}$.\\
$\bullet$ $D(\alpha_1,\cdots,\alpha_N)=\{f\in C_b^2(G^*) : \displaystyle\sum_{i=1}^{N}\alpha_i f_i'(0+)=0\}$.\\
For all $x\in G$, we define $\vec{e}(x)=\vec{e}_i$ if $x\in D_i, x\neq0$ (convention $\vec{e}(0)=\vec{e}_N$).
\noindent For $f\in C_b^2(G^*)$, $x\neq 0$, let $f'(x)$ be the derivative of $f$ at $x$ relatively to $\vec{e}(x)$ ($=f'_i(|x|)$ if $x\in D_i$) and $f''(x)=(f')'(x)$ ($=f''_i(|x|)$ if $x\in D_i$). We use the conventions $f'(0)=f_{N}'(0+), f''(0)=f_{N}''(0+)$. Now, associate to each ray $D_i$ a sign $\varepsilon_i\in\{-1,1\}$ and then define
$$ \varepsilon(x) = \begin{cases}
       \varepsilon_i & \text{if} \ x\in D_i,x\neq0 \\
       \varepsilon_N & \text{if} \ x=0 \\
   \end{cases}$$
To simplify, we suppose that $\varepsilon_1=\cdots=\varepsilon_p=1,\ \ \varepsilon_{p+1}=\cdots=\varepsilon_N=-1$ for some $p\leq N$. Set $$G^+=\bigcup_{1\leq i\leq p}D_i,\ \ G^-=\bigcup_{p+1\leq i\leq N}D_i.\ \textrm{Then}\ \ G=G^+\bigcup G^-.$$

\begin{figure}[h]
\begin{center}
\resizebox{9cm}{6cm}{\input{fig_graphs1.pstex_t}}
\caption{Graph $G$.}
\end{center}
\end{figure}

We also put $\alpha ^{+}=1-\alpha^-:=\sum_{i=1}^{p}\alpha_i$.
\begin{rem}
Our graph can be simply defined as $N$ pieces of $\R_+$ in which the $N$ origins are identified. The values of the $\vec{e}_i$ will not have any effect in the sequel.
\end{rem}

\end{defin}
\begin{defin} (Equation ($E$)).\\
On a probability space $(\Omega,\mathcal A,\mathbb P)$, let $W$ be a real white noise and $K$ be a stochastic flow of kernels on $G$ (a precise definition will be given in Section 2). We say that $(K,W)$ solves $(E)$ if for all $s\leq t, f\in D(\alpha_1,\cdots,\alpha_N), x\in G$,
$$K_{s,t}f(x)=f(x)+\int_s^tK_{s,u}(\varepsilon f')(x)W(du) + \frac{1}{2}\int_s^tK_{s,u}f''(x)du\ \ a.s.$$
If $K=\delta_\varphi$ is a solution of $(E)$, we simply say that $(\varphi,W)$ solves $(E)$.
\end{defin}
\begin{rems}
(1) If $(K,W)$ solves $(E)$, then $\sigma(W)\subset\sigma(K)$ (see Corollary \ref{mawloud}) below. So, one can simply say that $K$ solves $(E)$.\\
(2) The case $N=2,p=2, \varepsilon_1=\varepsilon_2=1$ (Figure 2) corresponds to Tanaka's SDE related to $SBM$ and includes in particular the usual Tanaka's SDE \cite{MR2235172}. In fact, let $(K^{\R},W)$ be a solution of (\ref{laurent}) with $\alpha=\alpha_1, \varepsilon(y)=\textrm{sgn}(y)$ and define $\psi(y)=|y|(\vec{e}_1 1_{y\geq0}+\vec{e}_2 1_{y<0}), y\in\R$. For all $x\in G$, define $K^G_{s,t}(x)=\psi(K^{\R}_{s,t}(y))$ with $y=\psi^{-1}(x)$. Let $f\in D(\alpha_1,\alpha_2), x\in G$ and $g$ be defined on $\R$ by $g(z)=f(\psi(z))$ $(g\in D_{\alpha_1})$. Since $K^{\R}$ satisfies(\ref{laurent}) in $(g,\psi^{-1}(x))$ ($g$ is the test function and $\psi^{-1}(x)$ is the starting point), it easily comes that $K^G$ satisfies $(E)$ in $(f,x)$. Similarly, if $K^G$ solves $(E)$, then $K^{\R}$ solves (\ref{laurent}).
\vspace{0.3cm}

\begin{figure}[h]
\begin{center}
%\resizebox{9cm}{0.7cm}{\input{fig_graphs2}}
\includegraphics[width=9cm]{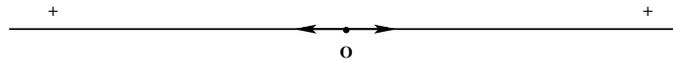}
\caption{Tanaka's SDE.}
\end{center}
\end{figure}

\noindent(3) As in (2), the case $N=2,p=1, \varepsilon_1=1, \varepsilon_2=-1$ (Figure 3) corresponds to (\ref{10}).
\begin{figure}[h]
\begin{center}
\includegraphics[width=9cm]{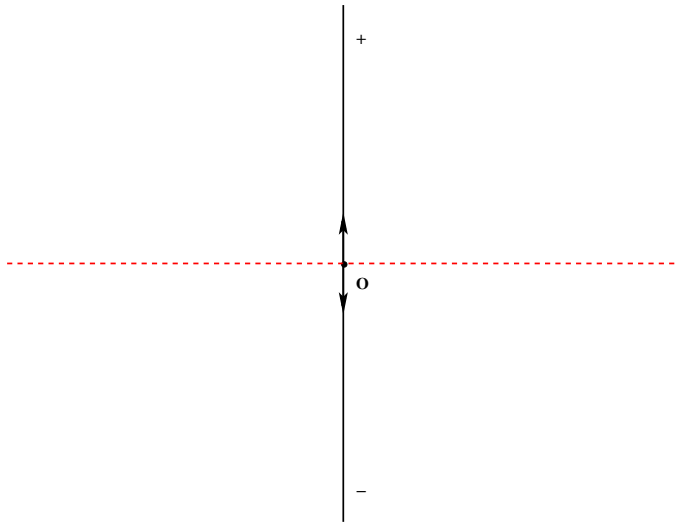}
\caption{SBM equation.}
\end{center}
\end{figure}

%\begin{figure}[h]
%\begin{center}
%\resizebox{9cm}{6cm}{\input{fig_graphs4.pstex_t}}
%\caption{Graph $G$.}
%\end{center}
%\end{figure}

\end{rems}
\noindent In this paper, we classify all solutions of $(E)$ by means of probability measures. We now state the first
\begin{thm}\label{ghasseni}
Let $W$ be a real white noise and $X_{t}^{s,x}$ be the flow associated to (\ref{10}) with $\alpha=\alpha^+$. Define $Z_{s,t}(x)=X_{t}^{s,\varepsilon(x)|x|}, s\leq t, x\in G$ and
\begin{eqnarray}
K_{s,t}^{W}(x)&=&\delta_{x+\vec{e}(x)\varepsilon(x)W_{s,t}}1_{\{t\leq \tau_{s,x}\}}\nonumber\\
&+&\big(\sum_{i=1}^{p}\frac{\alpha_i}{\alpha^+}\delta_{\vec{e}_i|Z_{s,t}(x)|} 1_{\{Z_{s,t}(x)>0\}}+\sum_{i=p+1}^{N}\frac{\alpha_i}{\alpha^-}\delta_{\vec{e}_i|Z_{s,t}(x)|}1_{\{Z_{s,t}(x)\leq0\}}\big)1_{\{t> \tau_{s,x}\}},\nonumber\
\end{eqnarray}
where $\tau_{s,x}=\inf\{r\geq s : x+\vec{e}(x)\varepsilon(x)W_{s,r}=0\}$. Then, $K^W$ is the unique Wiener solution of $(E)$. This means that $K^W$ solves $(E)$ and if $K$ is another Wiener solution of $(E)$, then for all $s\leq t, x\in G$,    \ \ $K^W_{s,t}(x)=K_{s,t}(x) \ a.s.$ 
\end{thm}
\noindent The proof of this theorem follows \cite{MR1905858} (see also \cite{MR0000004} for more details) with some modifications adapted to our case. We will use Freidlin-Sheu formula for Walsh Brownian motion to check that $K^W$ solves $(E)$. Unicity will be justified by means of the Wiener chaos decomposition (Proposition \ref{ghassen}). Besides the Wiener flow, there are also other weak solutions associated to $(E)$ which are fully described by the following  

\begin{thm}\label{hajri} (1) Define $$\Delta_k=\big\{u=(u_1,\cdots,u_k)\in [0,1]^k : {\sum_{i=1}^k }u_i=1\big\},\ \ k\geq 1.$$
Suppose $\alpha ^{+}\neq\frac{1}{2}$.\\
 (a) Let $m^+$ and $m^-$ be two probability measures respectively on $\Delta_p$ and $\Delta_{N-p}$ satisfying :
\begin{itemize}
\item[(+)]$\displaystyle\int_{\Delta_p} u_i m^+ (du)=\frac{\alpha _i}{\alpha^+}, \forall 1\leq i\leq p$,
\item[(-)]$\displaystyle\int_{\Delta_{N-p}}u_j m^- (du)=\frac{\alpha _{j+p}}{\alpha^-},\ \forall 1\leq j\leq N-p.$
\end{itemize}
\vspace{0.3cm}
Then, to $(m^+, m^-)$ is associated a stochastic flow of kernels $K^{m^+, m^-}$ solution of $(E)$.
\begin{itemize}
\item To \hspace{0.2cm} $(\delta_{(\frac{\alpha_1}{\alpha^+},\cdots,\frac{\alpha_p}{\alpha^+})},\delta_{(\frac{\alpha_{p+1}}{\alpha^-},\cdots,\frac{\alpha_N}{\alpha^-})})$ is associated a Wiener solution $K^{W}$.
\item To \hspace{0.2cm}$(\displaystyle{\sum_{i=1}^{p}} \frac{\alpha_i}{\alpha^+} \delta_{0,..,0,1,0,..,0},\displaystyle{\sum_{i=p+1}^{N}} \frac{\alpha_i}{\alpha^-} \delta_{0,..,0,1,0,..,0})$ is associated a coalescing stochastic flow of mappings $\varphi$.
\end{itemize}
\noindent (b) For all stochastic flow of kernels $K$ solution of $(E)$ there exists a unique pair of measures $(m^+,m^-)$ satisfying conditions $(+)$ and $(-)$ such that $K\overset{law}{=}K^{m^+,m^-}$.\\
\noindent (2) If $\alpha ^{+}=\frac{1}{2}, N>2$, there is just one solution of $(E)$ which is a Wiener solution.
\end{thm}
\begin{rems}
 (1) If $\alpha^+=1$, solutions of $(E)$ are characterized by a unique measure $m^+$ satisfying condition $(+)$ instead of a pair $(m^+,m^-)$ and a similar remark applies if $\alpha^-=1$.\\
(2) The case $\alpha^+=\frac{1}{2}, N=2$ does not appear in the last theorem since it corresponds to $dX_t=W(dt)$.
\end{rems}
This paper follows ideas of \cite{MR2235172} in a more general context and is organized as follows. In Section 2, we remind basic definitions of stochastic flows and Walsh Brownian motion. In Section 3, we use a \textquotedblleft specific\textquotedblright $SBM(\alpha^+)$ flow introduced by Burdzy-Kaspi  and excursion theory to construct all solutions of $(E)$. Unicity of solutions is proved in Section 4. Section 5 is an appendix devoted to Freidlin-Sheu formula stated in \cite{MR1743769} for a general class of diffusion processes defined by means of their generators. Here we first establish this formula using simple arguments and then deduce the characterization of Walsh Brownian motion by means of its generator (Proposition \ref{vi}).
\section{Stochastic flows and Walsh Brownian motion.}
Let $\mathcal{P}(G)$ be the space of probability measures on $G$ and $(f_n)_{n\in\N}$ be a sequence of functions dense in $\{f\in C_0(G), ||f||_{\infty}\leq 1\}$ with $C_0(G)$ being the space of continuous functions on $G$ which vanish at infinity. We equip $\mathcal{P}(G)$ with the distance $d(\mu,\nu)=(\sum_{n}2^{-n}(\int f_nd\mu-\int f_nd\nu)^2)^{\frac{1}{2}}$ for all $\mu$ and $\nu$ in $\mathcal{P}(G)$. Thus, $\mathcal{P}(G)$ is a locally compact separable metric space. Let us recall that a kernel $K$ on $G$ is a measurable mapping from $G$ into $\mathcal{P}(G)$. We denote by $E$ the space of all kernels on $G$ and we equip $E$ with the $\sigma$-field $\mathcal{E}$ generated by the mappings $K\longmapsto\mu K, \mu\in\mathcal{P}(G)$, with $\mu K$ the probability measure defined as $\mu K(A)=\int_{G}K(x,A)\mu (dx)$ for every $\mu\in\mathcal{P}(G)$. Let us recall some fundamental definitions from \cite{MR2060298}.
\subsection{Stochastic flows.}
Let $(\Omega,\mathcal{A},\mathbb P)$ be a probability space.
\begin{defin}{(Stochastic flow of kernels.)}\label{ghazel}
A family of $(E,\mathcal{E})$-valued random variables $(K_{s,t})_{s\leq t}$ is called a stochastic flow of kernels on $G$ if, $\forall s\leqslant t$ the mapping
$$\begin{array}{ccrcc}
K_{s,t} & : &{(G\times \Omega ,\mathcal{B}(G)\otimes \mathcal{A})}& \longrightarrow & {(\mathcal{P}(G),\mathcal{B}(\mathcal{P}(G)))} \\
& &(x,\omega)&\longmapsto & K_{s,t}(x,\omega)\\
\end{array}$$
 is measurable and if it satisfies the following properties:
\begin{enumerate}
\item $\forall s< t< u ,x\in G\ \ a.s. \ \ \forall f\in C_0(G),K_{s,u}f(x)=K_{s,t}(K_{t,u}f)(x)$ (flow property).
\item$\forall s\leqslant t$ the law of $K_{s,t}$ only depends on $t-s$.
\item For all $t_1< t_2<\cdots<t_n$, the family $\{K_{t_i,t_{i+1}},1\leq i \leq n-1\}$ is independent.
\item $\forall t\geq 0,x\in G,f\in C_0(G)$, $\lim\limits_{\substack{y \to x}} E[(K_{0,t}f(x)-K_{0,t}f(y))^2]=0$.
\item $\forall t\geq 0,f\in C_0(G)$, $\lim\limits_{\substack{x \to +\infty}}E[(K_{0,t}f(x))^2]=0.$
\item  $\forall x\in G,f\in C_0(G)$, $\lim\limits_{\substack{t \to 0+}} E[(K_{0,t}f(x)-f(x))^2]=0.$
\end{enumerate}
\end{defin}
\begin{defin}{(Stochastic flow of mappings.)}\label{york}
A family $(\varphi_{s,t})_{s\leq t}$ of measurable mappings from $G$ into $G$ is called a stochastic flow of mappings on $G$ if $K_{s,t}(x):=\delta_{\varphi_{s,t}(x)}$ is a stochastic flow of kernels on $G$. 
\end{defin}
\begin{rem} Let $K$ be a stochastic flow of kernels on $G$ and set $P^n_t=E[K_{0,t}^{\otimes n}], n\geq 1$. Then, $(P^n)_{n\geq 1}$ is a compatible family of Feller semigroups acting respectively on $C_0(G^n)$ (see Proposition 2.2 \cite{MR2060298}).
\end{rem}

\subsection{Walsh Brownian motion.}
Recall that for all $f\in C_0(G)$, $f_i$ is defined on $\R_+$. From now on, \textbf{we extend this definition on $\R$ by setting $f_i=0$ on $]-\infty,0[$}.  We will introduce Walsh Brownian motion $W(\alpha_1,\cdots,\alpha_N)$, by giving its transition density as defined in \cite{MR1022917}. On $C_0(G)$, consider
$$P_{t}f(h\vec{e}_j)=2\sum_ {i=1}^{N}\alpha_i p_tf_i(-h)+p_tf_j(h)-p_{t}f_j(-h), h>0,\ \ P_{t}f(0)=2\displaystyle\sum_ {i=1}^{N}\alpha_i p_tf_i(0).$$
where $(p_{t})_{t>0}$ is the heat kernel of the standard one dimensional Brownian motion. Then $(P_{t})_{t\geq 0}$ is a Feller semigroup on $C_0(G)$. A strong Markov process $Z$ with state space $G$ and semigroup $P_t$, and such that $Z$ is c\`adl\`ag
is by definition the Walsh Brownian motion $W(\alpha_1,\cdots,\alpha_N)$ on $G$. 
\subsubsection{Construction by flipping Brownian excursions.}\label{enri}
For all $n\geq0$, let $\mathbb D_n=\{\frac{k}{2^n},\ k\in\N\}$ and $\mathbb D=\cup_{n\in\N}\mathbb D_n$. For $0\leq u<v$, define $n(u,v)=\inf\{n\in\N : \mathbb D_n\cap ]u,v[\neq\emptyset\}$ and $f(u,v)=\inf \mathbb D_{n(u,v)}\cap]u,v[$.\\
Let $B$ be a standard Brownian motion defined on $(\Omega ,\mathcal A,\mathbb P)$ and $(\vec\gamma_r, r\in \mathbb D)$ be a sequence of independent random variables with the same law $\displaystyle\sum_{i=1}^{N}{\alpha_i}\delta_{\vec{e}_{i}}$ which is also independent of $B$. We define $$B^+_t=B_t-\displaystyle\min_{u\in[0,t]}B_u,\  g_t=\sup\{r\leq t : B^+_r=0\},\  d_t=\inf\{r\geq t : B^+_r=0\},$$
and finally $Z_t=\vec\gamma_r B^+_t, r=f(g_t,d_t)\ \textrm{if}\ B^+_t>0,\ \ Z_t=0\ \ \textrm{if}\ \ B^+_t=0$.
Then we have the following
\begin{prop}
 $(Z_t, t\geq0)$ is an $W(\alpha_1,\cdots,\alpha_N)$ on $G$ started at 0.
\end{prop}
\begin{proof}
 We use these notations $$min_{s,t}=\displaystyle\min_{u\in[s,t]}B_u,\ \vec{e}_{0,t}=\vec{e}(Z_t), \mathcal F_s=\sigma(\vec{e}_{0,u},B_u;0\leq u\leq s).$$
Fix $0\leq s<t$ and denote by $E_{s,t}=\{min_{0,s}=min_{0,t}\} (=\{g_t\leq s\}\ a.s.)$. Let $h : G\longrightarrow{\R}$ be a bounded measurable function. Then

$$E[h(Z_t)|\mathcal F_s]=E[h(Z_t)1_{E_{s,t}}|\mathcal F_s]
+E[h(Z_t)1_{E_{s,t}^{c}}|\mathcal F_s],$$
with
$$E[h(Z_t)1_{E_{s,t}^{c}}|\mathcal F_s]=\displaystyle{\sum_{i=1}^{N}}E[h_i(B^+_{t})1_{\{g_t>s,\vec{e}_{0,t}=\vec{e}_i\}}|\mathcal F_s]=\displaystyle{\sum_{i=1}^{N}}\alpha_iE[h_i(B^+_{t})1_{\{g_t>s\}}|\mathcal F_s].$$

If $B_{s,r}=B_r-B_s$, then the density of $(\underset{r\in[s,t]}{\min}B_{s,r},B_{s,t})$ with respect to the Lebesgue measure is given by
$$g(x,y)=\frac{2}{\sqrt{2\pi(t-s)^{3}}}(-2x+y)\exp(\frac{-(-2x+y)^2}{2(t-s)})1_{\{y>x,x<0\}}\ \textrm{(\cite{MR1100000} page 28)}.$$
Since $(B_{s,r}, r\geq s)$ is independent of $\mathcal F_s$, we get
\begin{eqnarray}
E[h_i(B^+_t)1_{\{g_t>s\}}|\mathcal F_s]&=&E[h_{i}(B_{s,t}-\underset{r\in[s,t]}{\min}{B_{s,r}})1_{\{\underset{r\in[0,s]}{\min}{B_{s,r}}>\underset{r\in[s,t]}{\min}{B_{s,r}}\}}|\mathcal F_s]\nonumber\\
&=&\int_{{\R}}1_{\{-B_s^+>x\}}(\int_{{\R}}h_i(y-x)g(x,y)dy)dx\nonumber\\
&=&2\int_{{\R_+}}h_i(u)p_{t-s}(B_s^+,-u)du\ \ (u=y-x)\nonumber\
\end{eqnarray}
and so $E[h(Z_t)1_{E_{s,t}^{c}}|\mathcal F_s]
=2\displaystyle{\sum_{i=1}^{N}}\alpha_ip_{t-s}h_i(-B^+_s)$. On the other hand

$$E[h(Z_t)1_{E_{s,t}}|\mathcal F_s]=E[h(\vec{e}_{0,s}(B_t-min_{0,s}))1_{E_{s,t}\cap (B_t>min_{0,s})}|\mathcal F_s]$$
$$=E[h(\vec{e}_{0,s}(B_t-min_{0,s}))1_{\{B_t>min_{0,s}\}}|\mathcal F_s]-E[h(\vec{e}_{0,s}(B_t-min_{0,s}))1_{E_{s,t}^{c}\cap(B_t>min_{0,s})}|\mathcal F_s].$$
Obviously on $\{\vec{e}_{0,s}=\vec{e}_k\}$, we have
\begin{eqnarray}
 E[h(\vec{e}_{0,s}(B_t-min_{0,s}))1_{\{B_t>min_{0,s}\}}|\mathcal F_s]&=&E[h_k(B_{s,t}+B_s^{+})1_{\{B_{s,t}+B_s^{+}>0\}}|\mathcal F_s]\nonumber\\
&=&p_{t-s}h_k(B_s^+)\nonumber\
\end{eqnarray}

and

$$E[h(\vec{e}_{0,s}(B_t-min_{0,s}))1_{E_{s,t}^{c}\cap(B_t>min_{0,s})}|\mathcal F_s]=E[h_k(B_{s,t}+B_s^{+})1_{\{-B_s^{+}>\underset{r\in[s,t]}{\min}B_{s,r},B_{s,t}+B_s^{+}>0\}}|\mathcal F_s]$$
$$=\int_{{\R}}h_k(y+B_s^{+})
1_{\{y+B_s^+>0\}}\left (\int_{\R}1_{\{-B_s^+>x\}}g(x,y)dx\right) dy
=p_{t-s}h_k(-B_s^+).$$
As a result, $E[h(Z_t)|\mathcal F_s]=P_{t-s}h(Z_s)$ where $P$ is the semigroup of $W(\alpha_{1},\cdots,\alpha_N)$.
\end{proof}
\begin{prop}\label{9awed}Let $M = (M_n)_{n\geq 0}$ be a Markov chain on $G$ started at $0$ with stochastic matrix $Q$ given by
\begin{equation}\label{sirine}
Q(0,\vec{e}_{i})=\alpha_i,\ \ Q(n\vec{e}_{i},(n+1)\vec{e}_{i})=Q(n\vec{e}_{i},(n-1)\vec{e}_{i})=\frac{1}{2}\ \ \forall i\in [1,N],\ \ n\geq 1.
\end{equation}
Then, for all $0\leq t_1<\cdots<t_p$, we have
 $$(\frac{1}{2^n}M_{\lfloor 2^{2n}{t_1}\rfloor},\cdots,\frac{1}{2^n}M_{\lfloor 2^{2n}{t_p}\rfloor})\xrightarrow[\text{$n\rightarrow +\infty $}]{\text{law}}(Z_{t_1},\cdots,Z_{t_p}),$$
where $Z$ is an $W(\alpha_1,\cdots,\alpha_N)$ started at $0$.
\end{prop}
\begin{proof}
 Let $B$ be a standard Brownian motion and define for all $n\geq1$ : $T_0^n(B)=T_0^n(|B|)=0$ and for $k\geq0$
\begin{eqnarray}
T_{k+1}^n(B)&=&\inf\{r\geq T_{k}^n(B),|B_r-B_{T_{k}^n}|=\frac{1}{2^n}\},\nonumber\\
T_{k+1}^n(|B|)&=&\inf\{r\geq T_{k}^n(|B|),||B_r|-|B_{T_{k}^n}||=\frac{1}{2^n}\}.\nonumber\
\end{eqnarray}
Then, clearly $T_{k}^n(B)=T_{k}^n(|B|)$ and so $(T_{k}^n(|B|))_{k\geq0}\overset{law}{=}(T_{k}^n(B))_{k\geq0}$. It is known (\cite{MR0000000} page 31) that $\lim\limits_{\substack{n \to +\infty}}T_{\lfloor 2^{2n}t\rfloor}^n(B)=t$ a.s. uniformly on compact sets. Then, the result holds also for $T_{\lfloor 2^{2n}t\rfloor}^n(|B|)$. Now, let $Z$ be the $W(\alpha_1,\cdots,\alpha_N)$ started at $0$ constructed in the beginning of this section from the reflected Brownian motion $B^+$. Let $T^n_k=T^n_k(B^+)$ (defined analogously to $T^n_k(|B|)$) and $Z^n_k=2^nZ_{T^n_k}$. Then obviously $(Z^n_k, k\geq0)\overset{law}{=}M$ for all $n\geq0$. Since a.s. $t\longrightarrow Z_t$ is continuous, it comes that a.s. $\forall t\geq0, \lim_{n\rightarrow+\infty}\frac{1}{2^n}Z^n_{\lfloor 2^{2n}{t}\rfloor}=Z_t$.
\end{proof}

\subsubsection{Freidlin-Sheu formula.}
\begin{thm}\cite{MR1743769}\label{palestine}
Let  $(Z_{t})_{t\geq 0}$ be a $W(\alpha_1,\cdots,\alpha_N)$ on $G$ started at $z$ and let $X_t=|Z_t|$. Then\\
(i) \;$(X_{t})_{t\geq 0}$ is a reflecting Brownian motion started at $|z|$.  \\
(ii) \;$B_t=X_t-\tilde L_t(X)-|z|$ is a standard Brownian motion where$$\tilde L_t(X)=\displaystyle{\lim_{\varepsilon \rightarrow 0^+ }}\frac{1}{2\varepsilon }\int_0^t 1_{\{|X_u|\leq \varepsilon\}}du.$$
(iii) \;$\forall f\in C_b^2(G^*)$,
\begin{equation}\label{Najia}
f(Z_t)=f(z)+\int_{0}^{t}f'(Z_s)dB_s+\frac{1}{2}\int_{0}^{t}f''(Z_s)ds+(\sum_ {i=1}^{N}\alpha_if_{i}'(0+))\tilde L_t(X).
\end{equation}
 \end{thm}
\begin{rems}\label{Dominique}
(1) By taking $f(z)=|z|$ and applying Skorokhod lemma, we find the following analogous of (\ref{3alim}),$$|Z_t|=|z|+B_t-\min_{s\leq u\leq t}[(|z|+B_u)\wedge0].$$
 From this observation, when $\varepsilon_i=1$ for all $i\in [1,N]$, we call $(E)$, Tanaka's SDE related to $W(\alpha_1,\cdots,\alpha_N)$.\\
(2) For $N\geq3$, the filtration $(\mathcal F_t^Z)$ has the martingale representation property with respect to $B$ \cite{MR1022917}, but there is no Brownian motion $W$ such that $\mathcal F_t^Z=\mathcal F_t^W$ \cite{MR1487755}. 
 \end{rems}

\vspace{0.05cm}
\noindent Using this theorem, we obtain the following characterization of $W(\alpha_1,\cdots,\alpha_N)$ by means of its semigroup.
\begin{prop}\label{vi}
Let\\
$\bullet$ $D(\alpha_1,\cdots,\alpha_N)=\{f\in C_b^2(G^*) : \displaystyle\sum_{i=1}^{N}\alpha_i f_i'(0+)=0\}$.\\
$\bullet$ $Q=(Q_t)_{t\geq 0}$ be a Feller semigroup satisfying
$$Q_tf(x)=f(x)+\frac{1}{2} \int_ {0}^{t} Q_uf''(x)du \ \ \forall f\in D(\alpha_1,\cdots,\alpha_N).$$
Then, $Q$ is the semigroup of  $W(\alpha_1,\cdots,\alpha_N)$.
 \end{prop}
\noindent\begin{proof} Denote by $P$ the semigroup of $W(\alpha_1,\cdots,\alpha_N)$, $A'$ and $D(A')$ being respectively its generator and its domain on $C_0(G)$. If
\begin{equation}\label{wang}
 D'(\alpha_1,\cdots,\alpha_N)=\{ f\in C_0(G)\bigcap D(\alpha_1,\cdots,\alpha_N), f'' \in C_0(G)\},
\end{equation}
then it is enough to prove these statements:\\
(i)\ $\forall t> 0, \ \ \ P_t(C_0(G))\subset D'(\alpha_1,\cdots,\alpha_N)$.\\
(ii)\ $ D'(\alpha_1,\cdots,\alpha_N)\subset D(A')$\ \ and \ \ $A'f(x)=\frac{1}{2}f''(x)$ on $D'(\alpha_1,\cdots,\alpha_N)$.\\
(iii) $D'(\alpha_1,\cdots,\alpha_N)$ is dense in $C_0(G)$ for $||.||_{\infty}$. \\
(iv) If $R$ and $R'$ are respectively the resolvents of $Q$ and $P$, then
$$R_{\lambda}=R_{\lambda}^{'} \ \ \ \forall \ \ \lambda > 0 \ \ \textrm{on}\ \  D'(\alpha_1,\cdots,\alpha_N).$$
The proof of (i) is based on elementary calculations using dominated convergence, (ii) comes from (\ref{Najia}), (iii) is a consequence of (i) and the Feller property of $P$ (approximate $f$ by $P_{\frac{1}{n}}f$). To prove (iv),  let $A$ be the generator of $Q$ and fix $f\in D'(\alpha_1,\cdots,\alpha_N)$. Then, $R_{\lambda}f$ is the unique element of $D(A)$ such that $(\lambda I-A)(R_{\lambda}f)=f$. We have $R'_{\lambda}f\in D'(\alpha_1,\cdots,\alpha_N)$ by (i), $D'(\alpha_1,\cdots,\alpha_N)\subset D(A)$ by hypothesis. Hence $R'_{\lambda}f\in D(A)$ and since $A=A'$ on $D'(\alpha_1,\cdots,\alpha_N)$, we deduce that $R_{\lambda}f=R'_{\lambda}f$.
\end{proof}
\section{Construction of flows associated to $(E)$.}\label{houwa}
In this section, we prove $(a)$ of Theorem \ref{hajri} and show that $K^W$ given in Theorem \ref{ghasseni} solves $(E)$.
\subsection{Flow of Burdzy-Kaspi associated to SBM.}\label{mahmoud}
\subsubsection{Definition.}\label{laa}
We are looking for flows associated to the SDE (\ref{10}). The flow associated to $SBM(1)$ which solves $(\ref{10})$ is the reflected Brownian motion above $0$ given by
$$Y_{s,t}(x)=(x+W_{s,t})1_{\{t\leq \tau_{s,x}\}}+(W_{s,t}-\inf_{u\in[\tau_{s,x},t]}W_{s,u}) 1_{\{t>\tau_{s,x}\}},$$
where 
\begin{equation}\label{bn}
 \tau_{s,x}=\inf\{r\geq s : x+W_{s,r}=0\}.
\end{equation}
and a similar expression holds for the $SBM(0)$ which is the reflected Brownian motion below $0$. These flows satisfy all properties of the $SBM(\alpha), \alpha\in ]0,1[$ we will mention below such that the \textquotedblleft strong\textquotedblright flow property (Proposition \ref{jnoun}) and the strong comparison principle (\ref{fouda}). When $\alpha\in]0,1[$, we follow Burdzy-Kaspi \cite{MR2094439}. In the sequel, we will be interested in $SBM(\alpha^+)$ and so we suppose in this paragraph that $\alpha^+\notin\{0,1\}$.\\
 With probability 1, for all rationals $s$ and $x$ simultaneously, equation $(\ref{10})$ has a unique strong solution with $\alpha=\alpha^+$. Define$$Y_{s,t}(x)=\underset{\underset{u<s,x<X_{s}^{u,y}}{u,y\in {\Q}}}{\inf X_{t}^{u,y}},\ \ {L}_{s,t}(x)=\displaystyle{\lim_{\varepsilon \rightarrow 0^+ }}\frac{1}{2\varepsilon }\int_s^t 1_{\{|Y_{s,u}(x)|\leq \varepsilon\}}du .$$
Then, it is easy to see that a.s.
\begin{equation}\label{fouda}
Y_{s,t}(x)\leq Y_{s,t}(y)\ \ \forall s\leq t, x\leq y.
\end{equation}
This implies that $x\longmapsto Y_{s,t}(x)$ is increasing and c\`adl\`ag for all $s\leq t$ a.s.\\
According to \cite{MR2094439} (Proposition 1.1), $t\longmapsto Y_{s,t}(x)$ is H\"older continuous for all $s,x$ a.s. and with probability equal to $1$: $\forall s,x\in \R, Y_{s,\cdot}(x)\ \textrm{satisfies (\ref{10})}$. We first check that $Y$ is a flow of mappings and we start by the following flow property:
\begin{prop}\label{jnoun}
 $\forall\  t\geq s$ a.s.
$$Y_{s,u}(x)=Y_{t,u}(Y_{s,t}(x))\ \ \forall u\geq t, x\in\R\ .$$
\end{prop}
\noindent\begin{proof} It is known, since pathwise uniqueness holds for the SDE (\ref{10}), that for a fixed $s\leq t\leq u, x\in\R$, we have $Y_{s,u}(x)=Y_{t,u}(Y_{s,t}(x))\ \ a.s.$ (\cite{MR1472487} page 161). Now, using the regularity of the flow, the result extends clearly as desired.
\end{proof}
\noindent To conclude that $Y$ is a stochastic flow of mappings, it remains to show the following 
\begin{lemma}\label{vie}

$\forall t\geq s$, $x\in \R$, $f\in C_0(\R)$
$$\lim_{y \rightarrow x} E[(f(Y_{s,t}(x))-f(Y_{s,t}(y)))^2]=0.$$
\end{lemma}
\noindent\begin{proof} We take $s=0$. For $g\in C_0(\R^2)$, set $$P_t^{(2)}g(x)= E[g(Y_{0,t}(x_1), Y_{0,t}(x_2))],\ \ x=(x_1,x_2).$$  If $\varepsilon >0,\ \  f_\varepsilon (x,y)=1_{\{|x-y|\geq \varepsilon\}}$, then by Theorem 10 in \cite{MR2280299}, 
$P_t^{(2)}f_\varepsilon (x,y)\xrightarrow [{\text{$y\rightarrow x $}}]{}0$.\\
For all $f\in C_0(\R)$, we have
$$E[(f(Y_{0,t}(x))-f(Y_{0,t}(y)))^2]=P_t^{(2)} f^{{\otimes}^2}(x,x)+P_t^{(2)}f^{{\otimes}^2}(y,y)-2 P_t^{(2)} f^{{\otimes}^2}(x,y) .$$
To conclude the lemma, we need only to check that
$$\lim_{y \rightarrow x} P_t^{(2)}f(y)=P_t^{(2)}f(x),\ \forall x\in \R^2 ,\ \ f\in C_0(\R^2).\ \ $$

\noindent Let $f=f_1 \otimes f_2$ with $f_i\in C_0(\R)$, $x=(x_1,x_2), y=(y_1,y_2)\in\R^2$. Then
$$|P_t^{(2)}f(y)-P_t^{(2)}f(x)|\leq M \displaystyle{\sum_{k=1}^2} P_t^{(2)}(|1 \otimes f_k-f_k \otimes 1|)(y_k,x_k),$$
where $M>0$ is a constant. For all $\alpha >0$, $\exists \varepsilon >0$, $|u-v|<\varepsilon \Rightarrow \forall 1\leq k\leq 2 :\ \ |f_k(u)-f_k(v)|<\alpha $. As a result
$$|P_t^{(2)}f(y)-P_t^{(2)}f(x)|\leq 2M\alpha +2M\displaystyle{\sum_{k=1}^2}||f_k||_{\infty} P_t^{(2)}{f}_\varepsilon(x_k,y_k),$$ 
and we arrive at $\limsup_{y\rightarrow x}|P_t^{(2)}f(y)-P_t^{(2)}f(x)|\leq 2M\alpha$ for all $\alpha>0$ which means that $\lim_{y \rightarrow x} P_t^{(2)}f(y)=P_t^{(2)}f(x)$. Now this easily extends by a density argument for all $f\in C_0(\R^2)$.
\end{proof}
In the coming section, we present some properties related to the coalescence of $Y$ we will require in Section \ref{humain} to construct solutions of $(E)$.
\subsubsection{Coalescence of the Burdzy-Kaspi flow.}
In this section, we suppose $\frac{1}{2}<\alpha^+<1$. The analysis of the case $0<\alpha^+<\frac{1}{2}$ requires an application of symmetry. Define $$T_{x,y}=\inf\{r\geq 0,\ Y_{0,r}(x)=Y_{0,r}(y)\},\ \ x,y\in\R.$$ 
By the fundamental result of \cite{MR1837288}, $T_{x,y}<\infty\ \ a.s.$ for all $x,y\in\R$. Due to the local time, coalescence always occurs  in $0$;  $Y_{0,r}(x)=Y_{0,r}(y)=0$ if $r=T_{x,y}$. Recall the definition of $\tau_{s,x}$ from (\ref{bn}). Then $T_{x,y}>\sup(\tau_{0,x},\tau_{0,y})$ a.s. (\cite{MR1837288} page 203). Set
$$L_{t}^x=x+(2\alpha^{+}-1)L_{0,t}(x),\ \ U(x,y)=\inf\{z\geq y : L_{t}^{x}=L_{t}^{y}=z\ \textrm{for some}\ \ t\geq 0\}, y\geq x.$$
According to \cite{MR1880238} (Theorem 1.1), there exists $\lambda>0$ such that $$\forall u\geq y> 0,\ \ \mathbb P(U(0,y)\leq u)=(1-\frac{y}{u})^{\lambda}.$$
Thus for a fixed $0<\gamma< 1$, we get $\lim_{y\rightarrow 0+}\mathbb P(U(0,y)\leq y^{\gamma})=\lim_{y\rightarrow0+}(1-y^{1-\gamma})^{\lambda}=1$.\\
From Theorem 1.1 \cite{MR1880238}, we have $U(x,y)-x\overset{law}{=}U(0,y-x)$ for all $0<x<y$ and so
\begin{equation}\label{wilson}
 \lim_{y\rightarrow x+}\mathbb P(U(x,y)-x\leq (y-x)^{\gamma})=1,\ \forall x\geq0.\ \ 
\end{equation}
\begin{lemma}\label{abed}
For all $x\in\R$, we have $\lim_{y\rightarrow x} T_{x,y}=\tau_{0,x}\ $ in probability.
\end{lemma}
\noindent\begin{proof}
In this proof we denote $Y_{0,t}(0)$ simply by $Y_t$. We first establish the result for $x=0$. For all $t>0$, we have
$$\mathbb P(t\leq T_{0,y})\leq \mathbb P({L}_{0,t}(0)\leq {L}_{0,T_{0,y}}(0))= \mathbb P( L_{t}^{0}\leq U(0,y)
)$$
since $(2\alpha^{+}-1)L_{0,T_{0,y}}(0)=U(0,y)$. The right-hand side converges to $0$ as $y\rightarrow 0+$ by (\ref{wilson}). On the other hand, by the strong Markov property at time $\tau_{0,y}$ for $y<0$, $$G_t(y):=\mathbb P(t\leq T_{0,y})=\mathbb P(t\leq\tau_{0,y})+E[ 1_{\{t>\tau_{0,y}\}}G_{t-\tau_{0,y}}(Y_{\tau_{y}})].$$
 For all $\epsilon>0$, 
$$E[1_{\{t>\tau_{0,y}\}}G_{t-\tau_{0,y}}(Y_{\tau_{0,y}})]= E[1_{\{t-\tau_{0,y}>\epsilon\}}G_{t-\tau_{0,y}}(Y_{\tau_{0,y}})]+E[ 1_{\{0<t-\tau_{0,y}\leq \epsilon\}}G_{t-\tau_{0,y}}(Y_{\tau_{0,y}})]$$
$$ \ \ \ \ \ \ \ \ \ \leq  E[G_{\epsilon}(Y_{\tau_{0,y}})]+\mathbb P(0<t-\tau_{0,y}\leq\epsilon).$$
From previous observations, we have $Y_{\tau_{0,y}}>0\ a.s.$ for all $y<0$ and consequently $Y_{\tau_{0,y}}\longrightarrow0+$ as $y\rightarrow 0-$. Since $\lim_{z\rightarrow0+}G_{\epsilon}(z)=0$, by letting $y\rightarrow0-$ and using dominated convergence, then $\epsilon\rightarrow0$, we get $\limsup\limits_{\substack{y \to 0-}}G_{t}(y)=0$ as desired for $x=0$. Now, the lemma easily holds after remarking that
$$T_{x,y}-\tau_{0,x}\overset{law}{=}T_{0,y-x}\ \textrm{if}\ 0\leq x<y\ \ \textrm{and}\ \ T_{x,y}-\tau_{0,x}\overset{law}{=}T_{0,x-y}\ \textrm{if}\ x< y\leq 0.$$ 
\end{proof}
\noindent For $s\leqslant t, x\in\R$, define 
\begin{equation}\label{lao}
 g_{s,t}(x)=\sup\{ u\in[s,t] : Y_{s,u}(x)=0\}\ \ (\sup(\emptyset)=-\infty).
\end{equation}
We use Lemma \ref{abed} to prove
\begin{lemma}\label{sophie}
 Fix $s, x\in\R$. Then, a.s. for all $t>\tau_{s,x}$, there exists $(v,y)\in\Q^2$ such that $$v<g_{s,t}(x)\ \textrm{and} \ Y_{s,r}(x)=Y_{v,r}(y)\ \forall\ r\geq g_{s,t}(x).$$
\end{lemma}
\noindent\begin{proof}
 We prove the result for $s=0$ and first for $x=0$. Let $t>0$, then for all $\epsilon>0$
$$\mathbb P(\exists\ \eta>0 : Y_{0,t}(\eta)=Y_{0,t}(-\eta))\geq \mathbb P(T_{-\epsilon,\epsilon}\leq t).$$
From $\mathbb P(t<T_{-\epsilon,\epsilon})\leq \mathbb P(t<T_{0,\epsilon})+\mathbb P(t<T_{0,-\epsilon})$ and the previous lemma, we have
$\lim_{\epsilon\rightarrow0}\mathbb P(t<T_{-\epsilon,\epsilon})=0$ and therefore $\mathbb P(\exists\ \eta>0 : Y_{0,t}(\eta)=Y_{0,t}(-\eta))=1$. Choose $\epsilon>0$, such that $Y_{0,t}(\epsilon)=Y_{0,t}(-\epsilon)$ and let  $v\in ]0,T_{-\epsilon,\epsilon}[\cap\Q$. Then $Y_{0,v}(\epsilon)>Y_{0,v}(-\epsilon)$ and for any rational $y\in ]Y_{0,v}(-\epsilon),Y_{0,v}(\epsilon)[$, we have by (\ref{fouda})
$$Y_{v,u}(Y_{0,v}(-\epsilon))\leq Y_{v,u}(y)\leq Y_{v,u}(Y_{0,v}(\epsilon)),\ \forall u\geq v.$$
The flow property (Proposition \ref{jnoun}) yields $ Y_{0,u}(-\epsilon)\leq Y_{v,u}(y)\leq Y_{0,u}(\epsilon),\ \forall u\geq v$. So necessarily $Y_{0,r}(0)=Y_{v,r}(y),\ \forall r\geq g_{0,t}(0)$. For $x>0$ and $\epsilon$ small enough, we have 
$$\mathbb P(Y_{0,t}(x+\epsilon)>Y_{0,t}(x), t>\tau_{0,x})\leq\mathbb P(\tau_{0,x}<t<T_{x,x+\epsilon}).$$
This shows that $\lim_{\epsilon\rightarrow 0}\mathbb P(Y_{0,t}(x+\epsilon)>Y_{0,t}(x)|t>\tau_{0,x})=0$ by Lemma \ref{abed}. Similarly, for $\epsilon$ small
$$\mathbb P(Y_{0,t}(x-\epsilon)<Y_{0,t}(x), t>\tau_{0,x})\leq\mathbb P(\tau_{0,x}<t<T_{x-\epsilon,x}).$$
The right-hand side converges to $0$ as $\epsilon\rightarrow0$ by Lemma \ref{abed} and so\\
$\lim_{\epsilon\rightarrow0}\mathbb P(Y_{0,t}(x)>Y_{0,t}(x-\epsilon)|t>\tau_{0,x})=0$. Since
$$\{Y_{0,t}(x+\epsilon)>Y_{0,t}(x-\epsilon)\}\subset \{Y_{0,t}(x+\epsilon)>Y_{0,t}(x)\}\cup \{Y_{0,t}(x)>Y_{0,t}(x-\epsilon)\},$$
we get $\mathbb P(\exists \epsilon>0 : Y_{0,t}(x-\epsilon)=Y_{0,t}(x+\epsilon)|t>\tau_{0,x})=1$. Following the same steps as the case $x=0$, we show the lemma for a fixed $t$ a.s. Finally, the result easily extends almost surely for all $t$.
\end{proof}
\noindent We close this section by the
\begin{lemma}\label{yor}
With probability 1, for all $(s_1,x_1)\neq(s_2,x_2)\in \Q^2$ simultaneously\\
(i) $T_{s_1,s_2}^{x_1,x_2}:=\inf\{r\geq sup(s_1,s_2) : Y_{s_1,r}(x_1)=Y_{s_2,r}(x_2)\}<\infty$,\\
(ii) $T_{s_1,s_2}^{x_1,x_2}>\sup(\tau_{s_1,x_1},\tau_{s_2,x_2})$,\\
(iii) $Y_{s_1,T_{s_1,s_2}^{x_1,x_2}}(x_1)=Y_{s_2,T_{s_1,s_2}^{x_1,x_2}}(x_2)=0,$\\
(iv) $Y_{s_1,r}(x_1)=Y_{s_2,r}(x_2)\ \forall r\geq T_{s_1,s_2}^{x_1,x_2}.$
\end{lemma}
\noindent\begin{proof}
(i) is a consequence of Proposition \ref{jnoun}, the independence of increments and the coalescence of $Y$.(ii) Fix $(s_1,x_1)\neq(s_2,x_2)\in \Q^2$ with $s_1\leq s_2$. By the comparison principle (\ref{fouda}) and Proposition \ref{jnoun},  $Y_{s_1,t}(x_1)\geq Y_{s_2,t}(x_2)$ for all $t\geq s_2$ or $Y_{s_1,t}(x_1)\leq Y_{s_2,t}(x_2)$  for all $t\geq s_2$. Suppose for example that $0<z:=Y_{s_1,s_2}(x_1)<x_2$ and take a rational $r\in]z,x_2[$. Then $T_{s_1,s_2}^{x_1,x_2}>\tau_{s_2,z}\geq \tau_{s_1,x_1}$ and $ T_{s_1,s_2}^{x_1,x_2}\geq T_{s_2,s_2}^{r,x_2}>\tau_{s_2,x_2}$. (iii) is clear since coalescence occurs in $0$. (iv) is an immediate consequence of the pathwise uniqueness of (\ref{10}).
\end{proof}

\subsection{Construction of solutions associated to $(E)$.}\label{humain}

We now extend the notations given in Section \ref{enri}. For all $n\geq0$, let $\mathbb D_n=\{\frac{k}{2^n},\ k\in\Z\}$ and $\mathbb D$ be the set of all dyadic numbers: $\mathbb D=\cup_{n\in\N}\mathbb D_n$. For $u<v$, define $n(u,v)=\inf\{n\in\N : \mathbb D_n\cap ]u,v[\neq\emptyset\}$ and $f(u,v)=\inf \mathbb D_{n(u,v)}\cap]u,v[$. Denote by $G_{\Q}=\{x\in G : |x|\in\Q_+\}$. We also fix a bijection $\psi : \N\longrightarrow\Q\times G_{\Q}$ and set $(s_i,x_i)=\psi(i)$ for all $i\geq0$.
\subsubsection{Construction of a stochastic flow of mappings $\varphi$ solution of $(E)$.}\label{mavie}
Let $W$ be a real white noise and $Y$ be the flow of the $SBM(\alpha^+)$ constructed from $W$ in the previous section. We first construct $\varphi_{s,\cdot}(x)$ for all $(s,x)\in \Q\times G_{\Q}$ and then extend this definition for all $(s,x)\in\R\times G$. We begin by $\varphi_{s_0,\cdot}(x_0)$, then  $\varphi_{s_1,\cdot}(x_1)$ and so on. To define $\varphi_{s_0,\cdot}(x_0)$, we flip excursions of $Y_{s_0,\cdot}(\varepsilon (x_0)|x_0|)$ suitably. Then let $\varphi_{s_1,t}(x_1)$ be equal to $\varphi_{s_0,t}(x_0)$ if $Y_{s_0,t}(\varepsilon (x_0)|x_0|)=Y_{s_1,t}(\varepsilon (x_1)|x_1|)$. Before coalescence of $Y_{s_0,\cdot}(\varepsilon (x_0)|x_0|)$ and $Y_{s_1,\cdot}(\varepsilon (x_1)|x_1|)$, we define $\varphi_{s_1,\cdot}(x_1)$ by flipping excursions of $Y_{s_1,\cdot}(\varepsilon (x_1)|x_1|)$ independently of what happens to $\varphi_{s_0,\cdot}(x_0)$ and so on. In what follows, we translate this idea rigorously.
Let $\vec\gamma^+,\vec\gamma^-$ be two independent random variables on any probability space such that
\begin{equation}\label{chahid}
\vec\gamma^+\overset{law}{=}\displaystyle\sum_{i=1}^{p}\frac{\alpha_i}{\alpha^+}\delta_{\vec{e}_{i}},\ \ \ \vec\gamma^-\overset{law}{=}\displaystyle\sum_{j=p+1}^{N}\frac{\alpha_j}{\alpha^-}\delta_{\vec{e}_{j}}.
\end{equation}
Let $(\Omega ,\mathcal A,\mathbb P)$ be a probability space rich enough and $W=(W_{s,t}, s\leq t)$ be a real white noise defined on it.
\noindent For all $s\leq t, x\in G$, let $Z_{s,t}(x):=Y_{s,t}(\varepsilon (x)|x|)$ where $Y$ is the flow of Burdzy-Kaspi constructed from $W$ as in Section \ref{laa} if $\alpha^+\notin\{0,1\}$ (= the reflecting Brownian motion associated to (\ref{10}) if $\alpha^+\in\{0,1\}$).\\ We retain the notations $\tau_{s,x}, g_{s,t}(x)$ of the previous section (see (\ref{bn}) and (\ref{lao})). For $s\in\R, x\in G$ define, by abuse of notations
$$\tau_{s,x}=\tau_{s,\varepsilon(x)|x|},\ g_{s,\cdot}(x)=g_{s,\cdot}(\varepsilon(x)|x|)\ \textrm{and}\ \ d_{s,t}(x)=\inf\{r\geq t : Z_{s,r}(x)=0\}.$$ 
It will be convenient to set $Z_{s,r}(x)=\infty$ if $r<s$. For all $q\geq1, u_0,\cdots,u_q\in\R, y_0,\cdots,y_q\in G$ define
$$T_{u_0,\cdots,u_{q}}^{y_0,\cdots,y_{q}}=\inf\{r\geq \tau_{u_{q},y_{q}} : Z_{u_{q},r}(y_{q})\in \{Z_{u_i,r}(y_i), i\in [1,q-1]\}\}.$$
 Let $\{(\vec\gamma^+_{s_0,x_0}(r),\vec\gamma^-_{s_0,x_0}(r)), r\in\mathbb D\cap[s_0,+\infty[\}$ be a family of independent copies of $(\vec\gamma^+,\vec\gamma^-)$ which is independent of $W$. We define $\varphi_{s_0,\cdot}(x_0)$ by 
$$\varphi_{s_0,t}(x_0)= \begin{cases}
       x_0+\vec{e}(x_0)\varepsilon(x_0)W_{s_0,t}& \text{if} \ s_0\leq t\leq \tau_{s_0,x_0} \\
        \ 0 & \text{if} \ t>\tau_{s_0,x_0}, Z_{s_0,t}(x_0)=0\\ 
        \vec\gamma^+_{s_0,x_0}(f_0)|Z_{s_0,t}(x_0)|,\ \ f_0=f(g_{s_0,t}(x_0),d_{s_0,t}(x_0)) & \text{if} \ t>\tau_{s_0,x_0}, Z_{s_0,t}(x_0)>0\\
        \vec\gamma^-_{s_0,x_0}(f_0)|Z_{s_0,t}(x_0)|,\ \ f_0=f(g_{s_0,t}(x_0),d_{s_0,t}(x_0)) & \text{if} \ t>\tau_{s_0,x_0}, Z_{s_0,t}(x_0)<0
       \end{cases}$$ 
Now, suppose that $\varphi_{s_0,\cdot}(x_0),\cdots,\varphi_{s_{q-1},\cdot}(x_{q-1})$ are defined and let $\{(\vec\gamma^+_{s_q,x_q}(r),\vec\gamma^-_{s_q,x_q}(r)), r\in\mathbb D\cap[s_q,+\infty[\}$ be a family of independent copies of $(\vec\gamma^+,\vec\gamma^-)$ which is also independent of\\
 $\sigma\left(\vec\gamma^+_{s_i,x_i}(r),\vec\gamma^-_{s_i,x_i}(r), r\in\mathbb D\cap[s_i,+\infty[, 1\leq i\leq q-1, W\right)$. Since $T_{s_0,\cdots,s_{q}}^{x_0,\cdots,x_{q}}<\infty$, let $i\in[1,q-1]$ and $(s_i,x_i)$ such that $Z_{s_q,t_0}(x_q)=Z_{s_i,t_0}(x_i)$ with $t_0=T_{s_0,\cdots,s_{q}}^{x_0,\cdots,x_{q}}$. We define $\varphi_{s_q,\cdot}(x_q)$ by 
$$\varphi_{s_q,t}(x_q)= \begin{cases}
       x_q+\vec{e}(x_q)\varepsilon(x_q)W_{s_q,t}& \text{if} \ s_q\leq t\leq \tau_{s_q,x_q} \\
        \ 0 & \text{if} \ t>\tau_{s_q,x_q}, Z_{s_q,t}(x_q)=0\\ 
        \vec\gamma^+_{s_q,x_q}(f_q)|Z_{s_q,t}(x_q)|,\ \ f_q=f(g_{s_q,t}(x_q),d_{s_q,t}(x_q)) & \text{if} \ t\in[\tau_{s_{q}, x_{q}},t_0], Z_{s_q,t}(x_q)>0\\
        \vec\gamma^-_{s_q,x_q}(f_q)|Z_{s_q,t}(x_q)|,\ \ f_q=f(g_{s_q,t}(x_q),d_{s_q,t}(x_q))  & \text{if} \ t\in[\tau_{s_{q}, x_{q}},t_0], Z_{s_q,t}(x_q)<0\\
      \varphi_{s_i,t}(x_i) & \text{if} \ t\geq t_0\\
       \end{cases}$$ 
In this way, we construct $(\varphi_{s,\cdot}(x), s\in\Q, x\in G_{\Q})$.\\
Now, for all $s\in\R, x\in G$, let $\varphi_{s,t}(x)=x+\vec{e}(x)\varepsilon(x)W_{s,t}\ \textrm{if}\ \ s\leq t\leq \tau_{s,x}$. If $t>\tau_{s,x}$, then by Lemma (\ref{sophie}), there exist $v\in\Q, y\in G_{\Q}$ such that $v<g_{s,t}(x)\ \textrm{and} \ Z_{s,r}(x)=Z_{v,r}(y)\ \forall\ r\geq g_{s,t}(x)$. In this case, we define $\varphi_{s,t}(x)=\varphi_{v,t}(y)$.
Later, we will show that $\varphi$ is a coalescing solution of $(E)$.
\subsubsection{Construction of a stochastic flow of kernels $K^{m^+,m^-}$ solution of $(E)$.}\label{rayoum}
Let $m^+$ and $m^-$ be two probability measures respectively on $\Delta_p$ and $\Delta_{N-p}$. Let $\mathcal U^+,\mathcal U^-$ be two independent random variables on any probability space such that
\begin{equation}\label{maalaoui}
 \mathcal U^+\overset{law}{=}m^+,\ \mathcal U^-\overset{law}{=}m^-.
\end{equation}
Let $(\Omega ,\mathcal A,\mathbb P)$ be a probability space rich enough and $W=(W_{s,t}, s\leq t)$ be a real white noise defined on it. We retain the notation introduced in the previous paragraph for all functions of $W$. We consider a family $\{(\mathcal U^+_{s_0,x_0}(r),\mathcal U^-_{s_0,x_0}(r)), r\in\mathbb D\cap[s_0,+\infty[\}$ of independent copies of $(\mathcal U^+,\mathcal U^-)$ which is independent of $W$.\\
If $t>\tau_{s_0,x_0}$ and $Z_{s_0,t}(x_0)>0$ (resp. $Z_{s_0,t}(x_0)<0$), let 
$$U_{s_0,t}^+(x_0) =\mathcal U^{+}_{s_0,x_0}(f_0)\ \ (\textrm{resp.}\  U_{s_0,t}^-(x_0) =\mathcal U^{-}_{s_0,x_0}(f_0)),\ f_0=f(g_{s_0,t}(x_0),d_{s_0,t}(x_0)).$$
Write $U^{+}_{s_0,t}(x_0)=(U^{+,i}_{s_0,t}(x_0))_{1\leq i\leq p}$ (resp. $U^{-}_{s_0,t}(x_0)=(U^{-,i}_{s_0,t}(x_0))_{p+1\leq i\leq N}$) if $Z_{s_0,t}(x_0)>0, t>\tau_{s_0,x_0}$ (resp. $Z_{s_0,t}(x_0)<0, t>\tau_{s_0,x_0}$) and now define
$$K_{s_0,t}^{m^+,m^-}(x_0)= \begin{cases}
       \delta_{x_0+\vec{e}(x_0)\varepsilon(x_0)W_{s_0,t}}\ & \text{if} \ s_0\leq t\leq \tau_{s_0,x_0} \\
        \ \sum_{i=1}^{p}U^{+,i}_{s_0,t}(x_0)\delta_{\vec{e}_i|Z_{s_0,t}(x_0)|} & \text{if} \ t>\tau_{s_0,x_0}, Z_{s_0,t}(x_0)>0\\ 
        \sum_{i=p+1}^{N}U^{-,i}_{s_0,t}(x_0)\delta_{\vec{e}_i|Z_{s_0,t}(x_0)|} & \text{if} \ t>\tau_{s_0,x_0}, Z_{s_0,t}(x_0)<0\\
        \delta_0 & \text{if} \ t>\tau_{s_0, x_0}, Z_{s_0,t}(x_0)=0\\
       
        \end{cases}$$ 
Suppose that $K^{m^+,m^-}_{s_0,\cdot}(x_0),\cdots,K^{m^+,m^-}_{s_{q-1},\cdot}(x_{q-1})$ are defined and let $\{(\mathcal U^+_{s_q,x_q}(r),\mathcal U^-_{s_q,x_q}(r)), r\in\mathbb D\cap[s_q,+\infty[\}$ be a family of independent copies of $(\mathcal U^+,\mathcal U^-)$ which is also independent of\\
$\sigma\left(\mathcal U^+_{s_i,x_i}(r),\mathcal U^-_{s_i,x_i}(r), r\in\mathbb D\cap[s_i,+\infty[, 1\leq i\leq q-1, W\right)$. If $t>\tau_{s_q,x_q}$ and $Z_{s_q,t}(x_q)>0$ (resp. $Z_{s_q,t}(x_q)<0$), we define $U^{+}_{s_q,t}(x_q)=(U^{+,i}_{s_q,t}(x_q))_{1\leq i\leq p}$ (resp. $U^{-}_{s_q,t}(x_q)=(U^{-,i}_{s_q,t}(x_q))_{p+1\leq i\leq N}$) by analogy to $q=0$.
Let $i\in[1,q-1]$ and $(s_i,x_i)$ such that $Z_{s_q,t_0}(x_q)=Z_{s_i,t_0}(x_i)$ with $t_0=T_{s_0,\cdots,s_{q}}^{x_0,\cdots,x_{q}}$. Then, define 
$$K_{s_q,t}^{m^+,m^-}(x_q)= \begin{cases}
       \delta_{x_q+\vec{e}(x_q)\varepsilon(x_q)W_{s_q,t}}\ & \text{if} \ s_q\leq t\leq \tau_{s_q,x_q} \\
        \ \sum_{i=1}^{p}U^{+,i}_{s_q,t}(x_q)\delta_{\vec{e}_i|Z_{s_q,t}(x_q)|} & \text{if} \ t_0>t>\tau_{s_q,x_q}, Z_{s_q,t}(x_q)>0\\ 
        \sum_{i=p+1}^{N}U^{-,i}_{s_q,t}(x_q)\delta_{\vec{e}_i|Z_{s_q,t}(x_q)|} & \text{if} \ t_0>t>\tau_{s_q,x_q}, Z_{s_q,t}(x_q)<0\\
        \delta_0 & \text{if} \  t_0\geq t>\tau_{s_q, x_q}, Z_{s_q,t}(x_q)=0\\
        K^{m^+,m^-}_{s_i,t}(x_i) & \text{if} \ t>t_0\\
        \end{cases}$$ 
In this way, we construct $(K_{s,}^{m^+,m^-}(x), s\in\Q, x\in G_{\Q})$.\\
Now, for $s\in\R, x\in G$, let $K^{m^+,m^-}_{s,t}(x)=\delta_{x+\vec{e}(x)\varepsilon(x)W_{s,t}}\ \textrm{if}\ \ s\leq t\leq \tau_{s,x}$. If $t>\tau_{s,x}$, let $v\in\Q, y\in G_{\Q}$ such that $v<g_{s,t}(x)\ \textrm{and} \ Z_{s,r}(x)=Z_{v,r}(y)\ \forall\ r\geq g_{s,t}(x)$. Then, define $K^{m^+,m^-}_{s,t}(x)=K^{m^+,m^-}_{v,t}(y)$.\\
In the next section we will show that $K^{m^+,m^-}$ is a stochastic flow of kernels on $G$ which solves $(E)$.
\subsubsection{Construction of $(K^{m^+,m^-},\varphi)$ by filtering.}
 Let $m^+$ and $m^-$ be two probability measures as in Theorem \ref{hajri} and $(\vec\gamma^+,\mathcal U^+), (\vec\gamma^-,\mathcal U^-)$ be two independent random variables satisfying
$$\mathcal U^+=(\mathcal U^{+,i})_{1\leq i\leq p}\overset{law}{=}m^+,\ \ \mathcal U^-=(\mathcal U^{-,j})_{p+1\leq j\leq N}\overset{law}{=}m^-,$$
\begin{equation}\label{walid}
\mathbb P(\vec\gamma^+=\vec{e}_{i}|\mathcal U^{+})=\mathcal U^{+,i},\ \forall i\in[1,p],
\end{equation}
and 
\begin{equation}\label{walidb}
\mathbb P(\vec\gamma^-=\vec{e}_{j}|\mathcal U^{-})=\mathcal U^{-,j},\ \forall j\in[p+1,N].\end{equation}
Then, in particular $(\vec\gamma^+,\vec\gamma^-)$ and $(\mathcal U^+,\mathcal U^-)$ satisfy respectively (\ref{chahid}) and (\ref{maalaoui}).\\
On a probability space $(\Omega ,\mathcal A,\mathbb P)$  consider the following independent processes
\begin{itemize}
 \item $W=(W_{s,t}, s\leq t)$ a real white noise.
\item $\{(\vec\gamma^+_{s,x}(r),\mathcal U^+_{s,x}(r)), r\in\mathbb D\cap[s,+\infty[, (s,x)\in\Q\times G_{\Q}\}$ a family of independent copies of $(\vec\gamma^+,\mathcal U^+)$.
\item $\{(\vec\gamma^-_{s,x}(r),\mathcal U^-_{s,x}(r)), r\in\mathbb D\cap[s,+\infty[, (s,x)\in\Q\times G_{\Q}\}$ a family of independent copies of $(\vec\gamma^-,\mathcal U^-)$.
\end{itemize}
Now, let $\varphi$ and $K^{m^+,m^-}$ be the processes constructed in Sections \ref{mavie} and \ref{rayoum} respectively from $(\vec\gamma^+,\vec\gamma^-,W)$ and $(\mathcal U^+,\mathcal U^-,W)$.  
Let $\sigma(\mathcal U^{+},\mathcal U^-,W)$ be the $\sigma$-field generated by $\{\mathcal U^{+}_{s,x}(r), \mathcal U^{-}_{s,x}(r), r\in\mathbb D\cap [s,+\infty[, (s,x)\in \Q\times G_{\Q}\}$ and $W$. We then have the
\begin{prop}\label{100}
(i) For all measurable bounded function $f$ on $G$, $s\leq t\in\R, x\in G$, with probability 1, $$K_{s,t}^{m^+,m^-}f(x)=E[f(\varphi_{s,t}(x))|\sigma(\mathcal U^{+},\mathcal U^{-},W)].$$
(ii) For all $s,x$, with probability 1, $\forall t\geq s$ $$|\varphi_{s,t}(x)|=|Z_{s,t}(x)|,\ \ \ \varphi_{s,t}(x)\in G^+\Leftrightarrow Z_{s,t}(x)\geq0\ \textrm{and}\ \ \ \varphi_{s,t}(x)\in G^-\Leftrightarrow Z_{s,t}(x)\leq0.$$
(iii) For all $s,x\neq y$, with probability 1
$$t_0:=\inf\{r\geq s : \varphi_{s,r}(x)=\varphi_{s,r}(y)\}=\inf\{r\geq s : Z_{s,r}(x)=Z_{s,r}(y)=0\}$$
and $\varphi_{s,r}(x)=\varphi_{s,r}(y),\ \ \forall r\geq t_0.$
\end{prop}
\begin{proof} (i) comes from (\ref{walid}), (\ref{walidb}) and the definiton of our flows, (ii) is clear by construction. By (ii) coalescence of $\varphi_{s,\cdot}(x)$ and $\varphi_{s,\cdot}(y)$ occurs in $0$ and so (iii) is clear.
\end{proof}
\noindent Next we will prove that $\varphi$ is a stochastic flow of mappings on $G$. It remains to prove that properties (1) and (4) in the definition are satisfied. As in Lemma \ref{vie}, property (4) can be derived from the following
\begin{lemma}\label{abedi}
 $\forall t\geq s, \epsilon>0, x\in G$, we have
$$\lim\limits_{\substack{y \to x}}\mathbb P(d(\varphi_{s,t}(x),\varphi_{s,t}(y))\geq\epsilon)=0.$$

\end{lemma}
\noindent\begin{proof}
We take $s=0$. Notice that for all $z\in\R$, we have
$$Y_{0,t}(z)=z+W_t\ \textrm{if}\ \ 0\leq t\leq \tau_{0,z}.$$
Fix $\epsilon>0, x\in\ G^+\setminus\{0\}$ and $y$ in the same ray as $x$ with $|y|>|x|, d(y,x)\leq\frac{\epsilon}{2}$. Then $d(\varphi_{0,t}(x),\varphi_{0,t}(y))=d(x,y)\leq\frac{\epsilon}{2}$ for $0\leq t\leq \tau_{0,|x|}\wedge \tau_{0,|y|}$ ($=\tau_{0,|x|}$ in our case). By Proposition \ref{100} (iii), we have $\varphi_{0,t}(x)=\varphi_{0,t}(y)\ \textrm{if}\ t\geq T_{|x|,|y|}$. This shows that
$$\{d(\varphi_{0,t}(x),\varphi_{0,t}(y))\geq\epsilon\}\subset\{\tau_{0,|x|}<t<T_{|x|,|y|}\}\ a.s.$$
By Lemma \ref{abed}, $$\mathbb P(d(\varphi_{0,t}(x),\varphi_{0,t}(y))\geq\epsilon)\leq \mathbb P(\tau_{0,|x|}<t<T_{|x|,|y|})\rightarrow0\ \textrm{as}\ y\rightarrow x, |y|>|x|.$$
By the same way,
$$\mathbb P(d(\varphi_{0,t}(x),\varphi_{0,t}(y))\geq\epsilon)\leq \mathbb P(\tau_{0,|y|}<t<T_{|x|,|y|})\rightarrow0\ \textrm{as}\ y\rightarrow x, |y|<|x|.$$
The case $x\in G^{-}$ holds similarly.
\end{proof}
\noindent\begin{prop}
$\forall s< t< u, x\in G$:
$$\varphi_{s,u}(x)=\varphi_{t,u}(\varphi_{s,t}(x))\ \ a.s.$$
\end{prop}
\noindent\begin{proof} Set $y=\varphi_{s,t}(x)$. Then, with probability $1$, $\forall r\geq t,\ \  Y_{s,r}(\varepsilon(x)|x|)=Y_{t,r}(Y_{s,t}(\varepsilon(x)|x|))$ and so, a.s. $\forall r\geq t\ \  Z_{s,r}(x)=Z_{t,r}(y)$. All the equalities below hold a.s.\\
$\bullet$ 1st case: $u\leq\tau_{s,x}$. We have $\tau_{t,y}=\inf\{r\geq t,\  Z_{t,r}(y)=0\}=\inf\{r\geq t,\ Z_{s,r}(x)=0\}=\tau_{s,x}$. Consequently $u\leq\tau_{t,y}$ and $\varphi_{s,u}(x)=\vec{e}(x)|Z_{s,u}(x)|=\vec{e}(y)|Z_{t,u}(y)|=\varphi_{t,u}(y)=\varphi_{t,u}(\varphi_{s,t}(x))$.\\
$\bullet$ 2nd case: $t\leq \tau_{s,x}<u$. We still have $\tau_{t,y}=\tau_{s,x}$ and so $g_{t,u}(y)=g_{s,u}(x)$. It is clear by construction that: $\varphi_{s,u}(x)=\varphi_{t,u}(y)=\varphi_{t,u}(\varphi_{s,t}(x))$.\\
$\bullet$ 3rd case: $\tau_{s,x}< t,\tau_{t,y}\leq u$. Since $\tau_{t,y}$ is a common zero of $(Z_{s,r}(x))_{r\geq s}$ and $(Z_{t,r}(y))_{r\geq t}$ before $u$, it comes that $g_{t,u}(y)=g_{s,u}(x)$ and therefore $\varphi_{s,u}(x)=\varphi_{t,u}(y)=\varphi_{t,u}(\varphi_{s,t}(x))$.\\
$\bullet$ 4th case: $\tau_{s,x}< t,u<\tau_{t,y}$. In such a case, we have $\varphi_{t,u}(y)=\vec{e}(y)|Z_{t,u}(y)|=\vec{e}(y)|Z_{s,u}(x)|$. Since $r\longmapsto\ Z_{s,r}(x)$ does not touch $0$ in the interval $[t,u]$ and $\varphi_{s,t}(x)=y$, we easily see that $\varphi_{s,u}(x)=\vec{e}(y)|Z_{s,u}(x)|=\varphi_{t,u}(y)$.
\end{proof}
\begin{prop}\label{wajj}
$\varphi$ is a coalescing solution of $(E)$.
\end{prop}
\noindent\begin{proof} We use these notations: $Y_u:=Y_{0,u}(0),\ \varphi_u:=\varphi_{0,u}(0)$. We first show that $\varphi$ is an $W(\alpha_1,\cdots,\alpha_N)$ on $G$. Define for all $n\geq1$ : $T_0^n(Y)=0$,
\begin{eqnarray}
T_{k+1}^n(Y)&=&\inf\{r\geq T_{k}^n(Y),d(\varphi_{r},\varphi_{T_{k}^n})=\frac{1}{2^n}\}=\inf\{r\geq T_{k}^n(Y),|Y_r-Y_{T_{k}^n}|=\frac{1}{2^n}\}\nonumber\\
&=&\inf\{r\geq T_{k}^n(Y),||Y_r|-|Y_{T_{k}^n}||=\frac{1}{2^n}\}, k\geq 0.\nonumber\
\end{eqnarray}
Remark that $|Y|$ is a reflected Brownian motion and denote $T_{k}^n(Y)$ simply by $T_{k}^n$. From the proof of Proposition \ref{9awed}, $\lim\limits_{\substack{n \to +\infty}}\displaystyle\sup_{t\leq K}|T_{\lfloor 2^{2n}t\rfloor}^n-t|=0\ a.s.$ for all $K>0$. Set $\varphi_k^n=2^n\varphi_{T_{k}^n}$. Then, since almost surely $t\longrightarrow\varphi_t$ is continuous, a.s. $\forall t\geq 0, \lim\limits_{\substack{n \to +\infty}}\frac{1}{2^n}\varphi_{\lfloor 2^{2n}t\rfloor}^n=\varphi_{t}$. By Proposition \ref{9awed}, it remains to show that for all $n\geq0$, $(\varphi_k^n, k\geq0)$ is a Markov chain (started at 0) whose transition mechanism is described by (\ref{sirine}). If $Y_k^n=2^nY_{T_{k}^n}$, then, by the proof of Proposition \ref{9awed} (since SBM is a special case of $W(\alpha_1,\cdots,\alpha_N)$), for all $n\geq 0$, $(Y_k^n)_{k\geq1}$ is a Markov chain on $\mathbb Z$ started at $0$ whose law is described by $$Q(0,1)=1-Q(0,-1)=\alpha^+,\ \ Q(m,m+1)=Q(m,m-1)=\frac{1}{2} \ \forall m\neq0.$$
Let $k\geq1$ and $x_0,..,x_k\in G$ such that $x_0=x_k=0$ and $|x_{h+1}-x_h|=1$ if $h\in[0,k-1]$. We write 
$$\{x_h, x_h=0, h\in [1,k]\}=\{x_{i_0},..,x_{i_q}\},\ i_0=0<i_1<\cdots<i_q=k$$
and 
$$\{x_h, x_h\neq0, h\in[1,k]\}=\{x_h\}_{h\in[i_0+1,i_1-1]}\cup\cdots\cup \{x_h\}_{h\in[i_{q-1}+1,i_{k-1}]}.$$
Assume that $$\{x_h\}_{h\in[i_0+1,i_1-1]}\subset D_{j_0},\cdots, \{x_h\}_{h\in[i_{q-1}+1,i_{k-1}]}\subset D_{j_{q-1}}$$
and define $$A_h^n=(Y_h^n=\varepsilon(x_h)|x_h|),\ \ E=(\vec{e}(\varphi_{i_0+1}^n)=\vec{e}_{j_0},\cdots,\vec{e}(\varphi_{i_{q-1}+1}^n)=\vec{e}_{j_{q-1}}).$$ 
If $i\in [1,p]$, we have
$$(\varphi_{k+1}^n=\vec{e}_{i}, \varphi_{k}^n=x_k,\cdots,\varphi_{0}^n=x_0)=\displaystyle\bigcap_{h=0}^kA_h^n\bigcap (Y_{k+1}^n-Y_k^n=1)\bigcap E\bigcap (\vec{e}(\varphi_{k+1}^n)=\vec{e}_i)$$
 and $(\varphi_{k}^n=x_k,\cdots,\varphi_{0}^n=x_0)=\displaystyle\bigcap_{h=0}^kA_h^n\bigcap E$. Now 
$$\mathbb P(\varphi_{k+1}^n=\vec{e}_{i}|\varphi_0^n=x_0,\cdots,\varphi_k^n=0)=\frac{\alpha_i}{\alpha^+}\mathbb P(Y_{k+1}^n-Y_k^n=1|Y_k^n=0)=\alpha_i.$$
Obviously, the previous argument can be applied to show that the transition probabilities of $(\varphi^n_{k},k\geq0)$ are given by (\ref{sirine}) and so $\varphi$ is an $W(\alpha_1,\cdots,\alpha_N)$ on $G$ started at $0$. Using (\ref{Najia}) for $\varphi$, it follows that $\forall f\in D(\alpha_1,\cdots\alpha_N)$,
$$f(\varphi_{t})=f(0)+\int_{0}^{t}f'(\varphi_{s})dB_s+\frac{1}{2}\int_{0}^{t}f''(\varphi_{s})ds$$
where $$B_t=|\varphi_t|-\tilde L_t(|\varphi|)=|Y_t|-\tilde L_t(|Y|)=\int_{0}^{t}\widetilde{\textrm{sgn}}(Y_s)dY_s$$ 
by Tanaka's formula for symmetric local time. But $Y$ solves (\ref{10}) and therefore $\int_{0}^{t}\widetilde{\textrm{sgn}}(Y_s)dY_s=\int_{0}^{t}\widetilde{\textrm{sgn}}(Y_s)W(ds)$. Since a.s. $\widetilde{\textrm{sgn}}(Y_s)=\varepsilon(\varphi_s)$ for all $s\geq0$, it comes that $\forall f\in D(\alpha_1,\cdots\alpha_N)$,
$$f(\varphi_{0,t}(x))=f(x)+\int_{0}^{t}f'(\varphi_{0,s}(x))\varepsilon(\varphi_{0,s}(x))W(ds)+\frac{1}{2}\int_{0}^{t}f''(\varphi_{0,s}(x))ds$$
when $x=0$. Finally, by distinguishing the cases $t\leqslant\tau_{0,x}$ and $t>\tau_{0,x}$, we see that the previous equation is also satisfied for $x\neq0$.
\end{proof}
\begin{corl}
$K^{m^+,m^-}$ is a stochastic flow of kernels solution of $(E)$.
\end{corl}
\noindent\begin{proof} By Proposition \ref{100} (i) and Jensen inequality, $K^{m^+,m^-}$ is a stochastic flow of kernels. The fact that  $K^{m^+,m^-}$ is a solution of $(E)$ is a consequence of the previous proposition and is similar to Lemma 4.6 \cite{MR2235172}.
\end{proof}
\begin{rems}\label{ta3}
\begin{enumerate}
\item[(i)] Define $\hat K_{s,t}(x,y)=K^{m^+,m^-}_{s,t}(x)\otimes \delta_{\varphi_{s,t}}(y)$. Then $\hat K$ is a stochastic flow of kernels on $G^2$.
\item[(ii)] If $(m^+,m^-)=(\delta_{(\frac{\alpha_1}{\alpha^+},\cdots,\frac{\alpha_p}{\alpha^+})},\delta_{(\frac{\alpha_{p+1}}{\alpha^-},\cdots,\frac{\alpha_N}{\alpha^-})})$, then
\begin{eqnarray}
K_{s,t}^{W}(x)&=&\delta_{x+\vec{e}(x)\varepsilon(x)W_{s,t}} 1_{\{t\leq \tau_{s,x}\}}\label{habla}\\
&+&\big(\sum_{i=1}^{p}\frac{\alpha_i}{\alpha^+}\delta_{\vec{e}_i|Z_{s,t}(x)|} 1_{\{Z_{s,t}(x)>0\}}+\sum_{i=p+1}^{N}\frac{\alpha_i}{\alpha^-}\delta_{\vec{e}_i|Z_{s,t}(x)|}1_{\{Z_{s,t}(x)\leq0\}}\big)1_{\{t> \tau_{s,x}\}}\nonumber\
\end{eqnarray}
is a Wiener solution of $(E)$.
\item[(iii)] If $(m^+,m^-)=\big(\displaystyle{\sum_{i=1}^{p}} \frac{\alpha_i}{\alpha^+} \delta_{(0,..,0,1,0,..,0)},\displaystyle{\sum_{i=p+1}^{N}} \frac{\alpha_i}{\alpha^-} \delta_{(0,..,0,1,0,..,0)}\big)$, then $K^{m^+,m^-}=\delta_\varphi$.
\end{enumerate}
\end{rems}
\section{Unicity of flows associated to $(E)$.}
Let $K$ be a solution of $(E)$ and fix $s\in\R, x\in G$. Then $(K_{s,t}(x))_{t\geq s}$ can be modified in such a way, a.s., the mapping $t\longmapsto K_{s,t}(x)$ is continuous from $[s,+\infty[$ into $\mathcal P(G)$. We will always consider this modification for $(K_{s,t}(x))_{t\geq s}$.
\begin{lemma}\label{z} Let $(K,W)$ be a solution of $(E)$. Then $\forall x\in G, s\in\R, a.s.$ $$K_{s,t}(x)=\delta_{x+\vec{e}(x)\varepsilon(x)W_{s,t}}, \ \textrm{if}\ \ s\leq t\leq \tau_{s,x}\ \textrm{where}\ \tau_{s,x}=\inf\{r\geq s,\ \varepsilon(x)|x|+W_{s,r}=0\}.$$
 \end{lemma}
\noindent\begin{proof} 
%(i) We fix $s=0$ and $x\in G$. We will check that with probability $1$, $$\forall f\in C_0(G),\ t\geq0,\ \lim\limits_{\substack{u \to t}}K_{0,u}f(x)=K_{0,t}f(x).$$
%Let $(g_n)_{n\geq1}$ be a sequence of functions dense in $C_0(G)$. Recall the definition of $D'(\alpha_1,\cdots,\alpha_N)$ from (\ref{wang}) and that $P_{\frac{1}{n}}g_n\in D'(\alpha_1,\cdots,\alpha_N)$ for all $n\geq 1$. Since $(K,W)$ satisfies $(E)$, a.s for all $n\geq1, t\geq0$
%$$\lim\limits_{\substack{u \to t}}K_{0,u}(P_{\frac{1}{n}}g_n)(x)=K_{0,t}(P_{\frac{1}{n}}g_n)(x).$$
%As $\{P_{\frac{1}{n}}g_n, n\geq1\}$ is dense in $C_0(G)$, it is easy to conclude.\\
We follow \cite{MR2235172} (Lemma 3.1). Assume that $x\neq0, x\in D_i$, $1\leq i\leq p$ and take $s=0$. Let $\beta_i=1$ and consider a set of numbers $(\beta_{j})_{1\leq j\leq N,j\neq i}$ such that $\displaystyle{\sum_{j=1}^{N}}\beta_{j}\alpha_j=0$. If $f(h\vec{e}_j) = \beta_jh$ for all $1\leq j\leq N$, then $f\in D(\alpha_1,\cdots,\alpha_N)$. Set $\tilde\tau_x=\inf\{r; K_{0,r}(x)(\cup_{j\neq i}D_j)>0\}$ and apply $f$ in $(E)$ to get 
\begin{equation}\label{ahh}
 \int_{D_i\setminus\{0\}}|y|K_{0,t}(x,dy)=|x|+W_t\ \textrm{for all}\ t\leq \tilde\tau_x.
\end{equation}
 By applying $f_k(y)=|y|^2 e^{\frac{-|y|}{k}},\ k\geq 1$ in $(E)$, we have for all $t\geq0$,
$$K_{0,t\wedge \tilde\tau_x}f_k(x)=f_k(x)+\int_0^t 1_{[0,\tilde\tau_x]}(u)K_{0,u}(\varepsilon f_k')(x)W(du) + \frac{1}{2}\int_0^{t\wedge \tilde\tau_x}K_{0,u}f_k''(x)du.$$
As $k\rightarrow\infty$, $K_{0,t\wedge \tilde\tau_x}f_k(x)$ tends to $\int_0^t |y|^2K_{0,t\wedge \tilde\tau_x}(x,dy)$ by monotone convergence. Let $A>0, xe^{-x}\leq A$ for all $x\geq0$. Since
$|f_{k}'(y)-2|y||\leq (4+A)|y|$,
$$\int_0^t 1_{[0,\tilde\tau_x]}(u)K_{0,u}(\varepsilon f_k')(x)W(du)\longrightarrow \int_0^t 1_{[0,\tilde\tau_x]}(u)\int_G 2|y|K_{0,u}(x,dy)W(du)$$
as $k\rightarrow\infty$ using (\ref{ahh}) and dominated convergence for stochastic integrals (\cite{MR1725357} page 142). From
$|f_k''(y)|\leq 2e^{\frac{-1}{k}|y|}+\frac{4+A}{k}|y|$, we get $\int_0^{t\wedge \tilde\tau_x}K_{0,u}f_k''(x)du\longrightarrow0$ as $k\rightarrow\infty$. By identifying the limits, we have
$$\int_{D_i\setminus\{0\}}(|y|-|x|-W_t)^2K_{0,t}(x,dy)=0\ \ \forall \ t\leq\tilde\tau_x.$$
This proves that for $t\leq\tilde\tau_x,\ K_{0,t}(x)=\delta_{x+\vec{e}(x)W_t}$. The fact that $\tau_{0,x}=\tilde\tau_x$ easily follows.
\end{proof}
\noindent The previous lemma entails the following
\begin{corl}\label{mawloud}
If $(K,W)$ is a solution of $(E)$, then $\sigma(W)\subset\sigma (K)$.
\end{corl}
\begin{proof}
 For all $x\in D_1$, we have $K_{0,t}(x)=\delta_{\vec{e_1}(|x|+W_{t})}$ if $t\leq\tau_{0,x}$. If $f$ is a positive function on $G$ such that $f_1(h)=h$, then $W_t=K_{0,t}f(x)-|x|$ for all $t\leq\tau_{0,x}, x\in D_1$. By considering a sequence $(x_k)_{k\geq0}$ converging to $\infty$, this shows that $\sigma(W_t)\subset\sigma(K_{0,t}(y), y\in D_1)$.
\end{proof}

\subsection{Unicity of the Wiener solution.}
\noindent In order to complete the proof of Theorem \ref{ghasseni}, we will prove the following
\begin{prop}\label{ghassen}
Equation $(E)$ has at most one Wiener solution: If $K$ and $K'$ are two Wiener solutions, then for all $s\leq t, x\in G, K_{s,t}(x)=K'_{s,t}(x)$ a.s. 
\end{prop}
\noindent\begin{proof} Denote by $P$ the semigroup of $W(\alpha_1,\cdots,\alpha_N)$, $A$ and $D(A)$ being respectively its generator and its domain on $C_0(G)$. Recall the definition of $D'(\alpha_1,\cdots,\alpha_N)$ from (\ref{wang}) and that
$$\forall t> 0 \ \ \ P_t(C_0(G))\subset D'(\alpha_1,\cdots,\alpha_N)\subset D(A)$$
(see Proposition \ref{vi}). Define
\\
$\mathcal S=\{f : G\longrightarrow{\R} : f, f',f''\in C_b(G^*)\  \textrm{and are extendable by continuity at 0 on each ray,}\ \ \ \ \\ \ \ \ \ \ \ \ \lim_{x\rightarrow\infty}f(x)=0\}$.\\
For $t>0, h$ a measurable bounded function on $G^{*}$, let $\lambda_t h(x)=2p_t h_j(|x|)$, if $x\in D_j$, where $h_j$ is the extension of $h_j$ that equals $0$ on $]-\infty,0]$. Then, the following identity can be easily checked using the explicit expression of $P$:
\begin{equation}\label{key}
(P_tf)'=-P_tf'+\lambda_t f'\ \ \textrm{on}\ G^*\ \ \textrm{for all}\ f\in\mathcal S.
\end{equation}
Fix $f\in\mathcal S$. We will verify that $(P_{t}f)'\in\mathcal S$. For $x=h\vec{e}_j\in G^*$, we have
$$(P_{t}f)'(x)=-2\sum_ {i=1}^{N}\alpha_i \int_{\R}f'_i(y-h)p_t(0,y)dy+\int_{\R}f'_j(y+h)p_t(0,y)dy+\int_{\R}f'_j(y-h)p_t(0,y)dy$$
 Clearly $(P_{t}f)'\in C_b(G^*)$ and is extendable by continuity at 0 on each ray. Furthermore, a simple integration by parts yields 
$$\int_{\R}f'_j(y+h)p_t(0,y)dy=C\int_{\R}f_j(y+h)yp_t(0,y)dy\ \textrm{for some}\ C\in\R$$
and since $\lim_{x\rightarrow\infty}f(x)=0$, we get $\lim_{x\rightarrow\infty}(P_tf)'(x)=0$. It is also easy to check that $(P_{t}f)'', (P_{t}f)''' \in C_b(G^*)$ and are extendable by continuity at 0 on each ray which shows that $(P_tf)'\in\mathcal S$.\\ 
Let $(K,W)$ be a stochastic flow that solves $(E)$ (not necessarily a Wiener flow) and fix $x=h\vec{e}_j\in G^*$. Our aim now is to establish the following identity
\begin{equation}\label{valerie}
K_{0,t}f(x)=P_{t}f(x)+\int_{0}^{t}K_{0,u}(D(P_{t-u}f))(x)W(du)
\end{equation}
where $Dg(x)=\varepsilon(x).g'(x)$. Note that $\int_{0}^{t}K_{0,u}(D(P_{t-u}f))(x)W(du)$ is well defined. In fact 
$$\int_{0}^{t}E[K_{0,u}(D(P_{t-u}f))(x)]^2du\leq \int_{0}^{t}P_u((D(P_{t-u}f))^2)(x)du\leq \int_{0}^{t}||(P_{t-u}f)'||^2_{\infty}du$$
and the right-hand side is bounded since (\ref{key}) is satisfied and $f'$ is bounded. Set $g=P_{\epsilon}f=P_{\frac{\epsilon}{2}}P_{\frac{\epsilon}{2}}f$. Then, since $P_{\frac{\epsilon}{2}}f\in C_0(G)$ $(\lim_{x\rightarrow\infty}P_{\frac{\epsilon}{2}}f(x)=0$ comes from $\lim_{x\rightarrow\infty}f(x)=0$), we have $g\in D'(\alpha_1,\cdots,\alpha_N)$. Now
$$K_{0,t}g(x)-P_{t}g(x)-\int_{0}^{t}K_{0,u}(D(P_{t-u}g))(x)W(du)=\displaystyle\sum_{p=0}^{n-1}(K_{0,\frac{{(p+1)t}}{n}}P_{t-\frac{{(p+1)t}}{n}}g-K_{0,\frac{{p}t}{n}}P_{t-\frac{{p}t}{n}}g)(x)$$
$$-\displaystyle\sum_{p=0}^{n-1}\int_{\frac{pt}{n}}^{\frac{(p+1)t}{n}}K_{0,u}D((P_{t-u}-P_{t-\frac{(p+1)t}{n}})g)(x)W(du)-\displaystyle\sum_{p=0}^{n-1}\int_{\frac{pt}{n}}^{\frac{(p+1)t}{n}}K_{0,u}D(P_{t-\frac{(p+1)t}{n}}g)(x)W(du).$$
For all $p\in\{0,..,n-1\}$, $g_{p,n}=P_{t-\frac{(p+1)t}{n}}g\in D'(\alpha_1,\cdots,\alpha_N)$ and so by replacing in $(E)$, we get
$$\int_{\frac{pt}{n}}^{\frac{(p+1)t}{n}}K_{0,u}Dg_{p,n}(x)W(du)=K_{0,\frac{(p+1)t}{n}}g_{p,n}(x)-K_{0,\frac{pt}{n}}g_{p,n}(x)-\int_{\frac{pt}{n}}^{\frac{(p+1)t}{n}}K_{0,u}Ag_{p,n}(x)du$$
$$=K_{0,\frac{(p+1)t}{n}}g_{p,n}(x)-K_{0,\frac{pt}{n}}g_{p,n}(x)-\frac{t}{n}K_{0,\frac{pt}{n}}Ag_{p,n}(x)-\int_{\frac{pt}{n}}^{\frac{(p+1)t}{n}}(K_{0,u}-K_{0,\frac{pt}{n}})Ag_{p,n}(x)du$$
Then we can write
$$K_{0,t}g(x)-P_{t}g(x)-\int_{0}^{t}K_{0,u}(D(P_{t-u}g))(x)W(du)=A_1(n)+A_2(n)+A_3(n),$$
where
$$A_1(n)=-\displaystyle\sum_{p=0}^{n-1}K_{0,\frac{pt}{n}}[P_{t-\frac{pt}{n}}g-P_{t-\frac{(p+1)t}{n}}g-\frac{t}{n}.AP_{t-\frac{(p+1)t}{n}}g](x),$$
$$A_2(n)=-\displaystyle\sum_{p=0}^{n-1}\int_{\frac{pt}{n}}^{\frac{(p+1)t}{n}}K_{0,u}D((P_{t-u}-P_{t-\frac{(p+1)t}{n}})g)(x)W(du),\ \ $$
$$A_3(n)=\displaystyle\sum_{p=0}^{n-1}\int_{\frac{pt}{n}}^{\frac{(p+1)t}{n}}(K_{0,u}-K_{0,\frac{pt}{n}})AP_{t-\frac{(p+1)t}{n}}g(x)du.$$
 Using $||K_{0,u}f||_{\infty}\leq ||f||_{\infty}$ if $f$ is a bounded measurable function, we obtain
$$|A_1(n)|\leq\displaystyle\sum_{p=0}^{n-1}||P_{t-\frac{(p+1)t}{n}}[P_{\frac{t}{n}}g-g-\frac{t}{n}.Ag]||_{\infty}
\leq n||P_{\frac{t}{n}}g-g-\frac{t}{n}.Ag||_{\infty},$$
with
$$n||P_{\frac{t}{n}}g-g-\frac{t}{n}.Ag||_{\infty}=t.||\frac{P_{t_n}g-g}{t_n}-Ag||_{\infty}\ (t_n:=\frac{t}{n}).$$
Since $g\in D(A)$, this shows that $A_1(n)$ converges to $0$ as $n\rightarrow\infty$. Note that $A_2(n)$ is the sum of orthogonal terms in $L^2(\Omega)$. Consequently
$$||A_2(n)||^2_{L^2(\Omega)}=\displaystyle\sum_{p=0}^{n-1}||\int_{\frac{pt}{n}}^{\frac{(p+1)t}{n}}K_{0,u}D((P_{t-u}-P_{t-\frac{(p+1)t}{n}})g)(x)W(du)||^2_{L^2(\Omega)}.$$
By applying Jensen inequality, we arrive at $$||A_2(n)||^2_{L^2(\Omega)}\leq\displaystyle\sum_{p=0}^{n-1}\int_{\frac{pt}{n}}^{\frac{(p+1)t}{n}}P_{u}V^2_u(x)du$$
where $V_u=(P_{t-u}g)'-(P_{t-\frac{(p+1)t}{n}}g)'$. By (\ref{key}), one can decompose $V_u$ as follows:
$$V_u=X_u+Y_u;\ \ X_u=-P_{t-u}g'+P_{t-\frac{(p+1)t}{n}}g',\ \  Y_u=\lambda_{t-u}g'- \lambda_{t-\frac{(p+1)t}{n}}g'$$
Using the trivial inequality $(a+b)^2\leq 2a^2+2b^2$, we obtain: $P_{u}V^2_u(x)\leq 2P_uX^2_u(x)+2P_uY^2_u(x)$ and so $$||A_2(n)||^2_{L^2(\Omega)}\leq 2B_1(n)+2B_2(n)$$
where $B_1(n)=\displaystyle\sum_{p=0}^{n-1}\int_{\frac{pt}{n}}^{\frac{(p+1)t}{n}}P_{u}X^2_u(x)du,\ \ B_2(n)=\displaystyle\sum_{p=0}^{n-1}\int_{\frac{pt}{n}}^{\frac{(p+1)t}{n}}P_{u}Y^2_u(x)du$.\\
If $p\in[0,n-1]$ and $u\in [{\frac{pt}{n}},{\frac{(p+1)t}{n}}]$, then $P_uX^2_u(x)\leq P_{u+t-\frac{p+1}{n}t}(g'-P_{\frac{p+1}{n}t-u}g')^2(x)$. The change of variable $v=(p+1)t-nu$ yields
\begin{eqnarray}
B_1(n)&\leq&\int_{0}^{t}P_{t-\frac{v}{n}}(P_{\frac{v}{n}}g'-g')^2(x)dv\nonumber\\
&\leq&\int_{0}^{t}(P_{t}g'^2(x)-2P_{t-\frac{v}{n}}(g'P_{\frac{v}{n}}g')(x)+P_{t-\frac{v}{n}}g'^2(x))dv.\nonumber\
\end{eqnarray}
By writing $P_{t-\frac{v}{n}}(g'P_{\frac{v}{n}}g')(x)$ as a function of $p$, we prove that $\lim_{n\rightarrow\infty}P_{t-\frac{v}{n}}(g'P_{\frac{v}{n}}g')(x)=P_{t}g'^2(x)$. Since $g'$ is bounded, by dominated convergence this shows that $B_1(n)$ tends to 0 as $n\rightarrow+\infty$. For $B_2(n)$, we write
$$P_{u}Y^2_u(x)=2\displaystyle\sum_{i=1}^{N}\alpha_ip_u((Y_u^2)_i)(-|x|)+p_u((Y_u^2)_j)(|x|)-p_u((Y_u^2)_j)(-|x|)$$
where $(Y_u)_i=2p_{t-u}g_i'-2p_{t-\frac{(p+1)t}{n}}g_i'$, defined on $\R_+^*$. It was shown before that this quantity tends to 0 as $n\rightarrow+\infty$ when $(p,g'_i)$ is replaced by $(P,g')$ in general and consequently $B_2(n)$ tends to 0 as $n\rightarrow+\infty$. Now

$$||A_3(n)||_{L^2(\Omega)}\leq \displaystyle\sum_{p=0}^{n-1}||\int_{\frac{pt}{n}}^{\frac{(p+1)t}{n}}(K_{0,u}-K_{0,\frac{pt}{n}})AP_{t-\frac{(p+1)t}{n}}g(x)du||_{L^2(\Omega)}.$$
Set $h_{p,n}=AP_{t-\frac{(p+1)t}{n}}g$. Then $h_{p,n}\in D'(\alpha_1,\cdots,\alpha_N)$ for all $p\in[0,n-1]$ (if $p=n-1$ remark that $h_{p,n}=P_{\frac{\epsilon}{2}}AP_{\frac{\epsilon}{2}}f$). By the Cauchy-Schwarz inequality   

$$||A_3(n)||_{L^2(\Omega)}\leq\sqrt{t}\left\{\sum_{p=0}^{n-1}\int_{\frac{pt}{n}}^{\frac{(p+1)t}{n}}E[((K_{0,u}-K_{0,\frac{pt}{n}})h_{p,n}(x))^2]du\right\}^{\frac{1}{2}}.$$
If $u\in [\frac{pt}{n},\frac{(p+1)t}{n}]$:
\begin{eqnarray}
E[((K_{0,u}-K_{0,\frac{pt}{n}})h_{p,n}(x))^2]&\leq& E[K_{0,\frac{pt}{n}}(K_{\frac{pt}{n},u}h_{p,n}-h_{p,n})^2(x)]\nonumber\
\\
&\leq&E[K_{0,\frac{pt}{n}}(K_{\frac{pt}{n},u}h_{p,n}^2-2h_{p,n}K_{\frac{pt}{n},u}h_{p,n}+h_{p,n}^2)(x)]\nonumber\
\\
&\leq&||P_{u-\frac{pt}{n}}h_{p,n}^2-2h_{p,n}P_{u-\frac{pt}{n}}h_{p,n}+h_{p,n}^2||_{\infty}\nonumber\
\\
&\leq&2||h_{p,n}||_{\infty}||P_{u-\frac{pt}{n}}h_{p,n}-h_{p,n}||_{\infty}+||P_{u-\frac{pt}{n}}h_{p,n}^2-h_{p,n}^2||_{\infty}.\nonumber\
\end{eqnarray}
Therefore $||A_3(n)||_{L^2(\Omega)}\leq \sqrt{t}(2C_1(n)+C_2(n))^{\frac{1}{2}}$, where
$$C_1(n)=\displaystyle\sum_{p=0}^{n-1}||h_{p,n}||_{\infty}\int_{\frac{pt}{n}}^{\frac{(p+1)t}{n}}||P_{u-\frac{pt}{n}}h_{p,n}-h_{p,n}||_{\infty}du, C_2(n)=\displaystyle\sum_{p=0}^{n-1}\int_{\frac{pt}{n}}^{\frac{(p+1)t}{n}}||P_{u-\frac{pt}{n}}h_{p,n}^2-h_{p,n}^2||_{\infty}du.$$
From $||h_{p,n}||_{\infty}\leq ||Ag||_{\infty}$ and $||P_{u-\frac{pt}{n}}h_{p,n}-h_{p,n}||_{\infty}\leq ||P_{u-\frac{pt}{n}}Ag-Ag||_{\infty}$, we get
$$C_1(n)\leq||Ag||_{\infty}\displaystyle\sum_{p=0}^{n-1}\int_{\frac{pt}{n}}^{\frac{(p+1)t}{n}}||P_{u-\frac{pt}{n}}Ag-Ag||_{\infty}du\leq||Ag||_{\infty}\int_{0}^{t}||P_{\frac{z}{n}}Ag-Ag||_{\infty}dz.$$
As $Ag\in C_0(G)$, $C_1(n)$ tends to $0$ obviously. On the other hand, $h_{p,n}^2\in D(\alpha_1,\cdots,\alpha_N)$ (this can be easily verified since $h_{p,n}$ is continuous and $\displaystyle\sum_{i=0}^{N}\alpha_i{(h_{p,n})}'_i(0+)=0$). We may apply (\ref{Najia}) to get\\
$C_2(n)=\frac{1}{n}\displaystyle\sum_{p=0}^{n-1}\int_{0}^{t}||P_{\frac{z}{n}}h_{p,n}^2-h_{p,n}^2||_{\infty}dz\leq \frac{1}{2n}\displaystyle\sum_{p=0}^{n-1}\int_{0}^{t}\int_{0}^{\frac{z}{n}}||(h_{p,n}^2)''||_{\infty}dudz$.\\
Now we verify that $h_{p,n}', h_{p,n}''$ are uniformly bounded with respect to $n$ and $0\leq p\leq n-1$. In fact $||h_{p,n}''||_{\infty}=||2Ah_{p,n}||_{\infty}\leq 2||AP_{\frac{\epsilon}{2}}f||_{\infty}$. Write $h_{p,n}=P_{t-\frac{p+1}{n}t+\frac{\epsilon}{2}}P_{\frac{\epsilon}{4}}AP_{\frac{\epsilon}{4}}f$ where $P_{\frac{\epsilon}{4}}AP_{\frac{\epsilon}{4}}f\in D'(\alpha_1,\cdots,\alpha_N)$. Then, by (\ref{key}), $||h_{p,n}'||_{\infty}$ is uniformly bounded with respect to $n, p\in[0,n-1]$ and so the same holds for $||(h_{p,n}^2)''||_{\infty}$. As a result $C_2(n)$ tends to $0$ as $n\rightarrow\infty$. Finally
$$K_{0,t}g(x)=P_{t}g(x)+\int_{0}^{t}K_{0,u}(D(P_{t-u}g))(x)W(du).$$
Now, let $\epsilon$ go to $0$, then $K_{0,t}g(x)$ tends to $K_{0,t}f(x)$ in $L^2(\Omega)$. Furthermore
\\
$||\displaystyle\int_{0}^{t}K_{0,u}(D(P_{t-u}g))(x)W(du)-\int_{0}^{t}K_{0,u}(D(P_{t-u}f))(x)W(du)||_{\L^2(\Omega)}^2$\\
$\leq\displaystyle\int_{0}^{t}P_u((P_{t-u}g)'-(P_{t-u}f)')^2(x)du.$\\
Using the derivation formula (\ref{key}), the right side may be decomposed as $I_{\epsilon}+J_{\epsilon}$, where 
$$I_{\epsilon}=\displaystyle\int_{0}^{t}P_u(P_{t-u}g'-P_{t-u}f')^2(x)du,\ \ J_{\epsilon}=\displaystyle\int_{0}^{t}P_u(\lambda_{t-u}g'-\lambda_{t-u}f')^2(x)du.$$
By Jensen inequality, $I_{\epsilon}\leq tP_t(g'-f')^2(x)$. Since $g'(y)=-P_{\epsilon}f'(y)+2\lambda_{\epsilon}f'(y)\longrightarrow f'(y)$ as $\epsilon\rightarrow0, P_t(x,dy)\ a.s.$, we get $I_{\epsilon}\longrightarrow0$ as $\epsilon\rightarrow0$ by dominated convergence.
\noindent Similarly $J_{\epsilon}$ tends to $0$ as $\epsilon\rightarrow0$. This establishes (\ref{valerie}). Now assume that $(K,W)$ is a Wiener solution of $(E)$ and let $f\in \mathcal S$. Since $K_{0,t}f(x)\in L^2(\mathcal F_{\infty}^{W_{0,\cdot}})$ , let $K_{0,t}f(x)=P_tf(x)+\sum_{n=1}^{\infty}J^n_tf(x)$
be the decomposition in Wiener chaos of $K_{0,t}f(x)$ in $L^2$ sense (\cite{MR1725357} page 202). By iterating (\ref{valerie}) (recall that $(P_tf)'\in\mathcal S$ ),  we see that for all $n\geq1$
$$J^n_tf(x)=\int_{0<s_1<\cdots<s_n<t}P_{s_1}(D(P_{s_2-s_1}\cdots D(P_{t-s_n}f)))(x)dW_{0,s_1}\cdots dW_{0,s_n}.$$
 If $K'$ is another Wiener flow satisfying (\ref{valerie}), then $K_{0,t}f(x)$ and $K'_{0,t}f(x)$ must have the same Wiener chaos decomposition for all $f\in\mathcal S$, that is $K_{0,t}f(x)=K'_{0,t}f(x)$ a.s. Consequently $K_{0,t}f(x)=K'_{0,t}f(x)$ a.s. for all $f\in D'(\alpha_1,\cdots,\alpha_N)$ since this last set is included in $\mathcal S$ and the result extends for all $f\in C_0(G)$ by a density argument. This completes the proof when $x\neq0$. The case $x=0$ can be deduced from property $(4)$ in the Definition \ref{ghazel}.
\end{proof}
\noindent\textbf{Consequence:} We already know  that $K^{W}$ given by (\ref{habla})
is a Wiener solution of $(E)$. Since $\sigma(W)\subset\sigma (K)$, we can define $K^{*}$ the stochastic flow obtained by filtering $K$ with respect to $\sigma(W)$ (Lemma 3-2 (ii) in \cite{MR2060298}). Then $\forall s\leq t, x\in G,\ \ K_{s,t}^{*}(x)=E[K_{s,t}(x)|\sigma(W)]\  a.s$. As a result, $(K^{*},W)$ solves also $(E)$ and by the last proposition, we have:
\begin{equation}\label{k}
\forall s\leq t, x\in G,\ \ E[K_{s,t}(x)|\sigma(W)]=K_{s,t}^{W}(x)\ a.s.
\end{equation}
\textbf{From now on, $(K,W)$ is a solution of $(E)$} defined on $(\Omega,\mathcal{A},\mathbb P)$. Let $P_t^{n}=E[K_{0,t}^{\otimes n}]$ be the compatible family of Feller semigroups associated to $K$. We retain the notations introduced in Section 3 for all functions of $W$ ($Y_{s,t}(x), Z_{s,t}(x), g_{s,t}(x)\cdots)$. In the next section, starting from $K$, we construct a flow of mappings $\varphi^c$ which is a solution of $(E)$. This flow will play an important role to characterize the law of $K$. 
\subsection{Construction of a stochastic flow of mappings solution of $(E)$ from $K$.}\label{tss}
Let $x\in G$, $t>0$. By (\ref{k}), on $\{t>\tau_{0,x}\}, K_{0,t}(x)$ is supported on $$\{|Z_{0,t}(x)|\vec{e}_i,\ \ 1\leq i\leq p\}\ \ \textrm{if}\ Z_{0,t}(x)>0$$
and is supported on $$\{|Z_{0,t}(x)|\vec{e}_i,\ \ p+1\leq i\leq N\}\ \ \textrm{if}\ Z_{0,t}(x)\leq0.$$
In \cite{MR2060298} (Section 2.6), the $n$ point motion $X^n$ started at $(x_1,\cdots,x_n)\in G^n$ and associated with $P^n$ has been constructed on an extension $\Omega\times \Omega'$ of $\Omega$ such that the law of $\omega'\longmapsto X^n_t(\omega,\omega')$ is given by $K_{0,t}(x_1,dy_1)\cdots K_{0,t}(x_n,dy_n)$. For each $(x,y)\in G^2$, let $(X^x_t,Y^y_t)_{t\geq0}$ be the two point motion started at $(x,y)$ associated with $P^{2}$ as preceded. Then $|X^x_t|=|Z_{0,t}(x)|, |Y^y_t|=|Z_{0,t}(y)|$ for all $t\geq0$ and so $$T^{x,y}:=\inf\{r\geq0,X^x_r=Y^y_r\}<+\infty\ \ \textrm{a.s}.$$
To $(P^{n})_{n\geq 1}$, we associate a compatible family of Markovian coalescent semigroups $(P^{n,c})_{n\geq 1}$  as described in \cite{MR2060298} (Theorem 4.1): Let $X^n$ be the $n$ point motion started at $(x_1,\cdots,x_n)\in G^n$. We denote the ith coordinate of $X_t^n$ by $X^n_t(i)$. Let $$T_1=\inf\{u\geq0, \exists i<j,\ X_u^n(i)=X_u^n(j)\},\ \ X_t^{n,c}:=X_t^{n}, t\in [0,T_1].$$
Suppose that $X_{T_1}^n(i)=X_{T_1}^n(j)$ with $i<j$. Then define the process 
$$X^{n,1}_t(h)=X^{n}_t(h)\ \textrm{for}\ \ h\neq j, X^{n,1}_t(j)=X^{n,1}_t(i), t\geq T_1.$$
Note that the ith coordinate of $X^{n,1}$ and the jth one are equal. Now set $$T_2=\inf\{u\geq T_1, \exists h<k, h\neq j, k\neq j, \ X_u^{n,1}(h)=X_u^{n,1}(k)\}.$$ For $t\in [T_1,T_2]$, we define $X_t^{n,c}=X_t^{n,1}$ and so on. In this way, we construct a Markov process $X^{n,c}$ such that for all $i, j\in [1,n]$, $X^{n,c}(i)$ and $X^{n,c}(j)$ meet after a finite time and then stick to gether. Let $P_t^{n,c}(x_1,\cdots,x_n,dy)$ be the law of $X_t^{n,c}$. Then we have:
\begin{lemma}\label{kokou}
$(P^{n,c})_{n\geq 1}$ is a compatible family of Feller semigroups associated with a coalescing flow of mappings $\varphi^c$.
\end{lemma}
\begin{proof}By Theorem $4.1$ \cite{MR2060298}, we only need to check that: $\forall t>0,\varepsilon>0,x\in G,$ $$\lim\limits_{\substack{y \to x}}\mathbb P(\{T^{x,y}>t\}\cap\{d(X^x_t,Y^y_t)>\varepsilon\})=0 \ \ (C).$$
As $|X^x_u|=|Z_{0,u}(x)|, |Y^y_u|=|Z_{0,u}(y)|$ for all $u\geq0$, we have $\{t<T^{x,y}\}\subset \{t<T_{\varepsilon(x)|x|,\varepsilon(y)|y|}\}$. For $y$ close to $x$, $\{d(X^x_t,Y^y_t)>\varepsilon\}\subset \{\inf(\tau_{0,x},\tau_{0,y})<t\}$. Now $(C)$ holds from Lemma \ref{abed}.
\end{proof}
\noindent\textbf {Consequence:} Let $\nu$ (respectively $\nu^{c}$) be the Feller convolution semigroup associated with $(P^{n})_{n\geq 1}$ (respectively $(P^{n,c})_{n\geq 1}$). By the proof of Theorem 4.2 \cite{MR2235172}, there exists a joint realization $(K^1,K^2)$ where $K^1$ and $K^2$ are two stochastic flows of kernels satisfying $K^1\overset{law}{=}\delta_{\varphi^c}$, $K^2\overset{law}{=}K$ and such that:
\begin{enumerate}
\item [(i)] $\hat K_{s,t}(x,y)=K^1_{s,t}(x)\otimes K^2_{s,t}(y)$ is a stochastic flow of kernels on $G^2$,
\item [(ii)] For all $s\leq t, x\in G,\ K^2_{s,t}(x)=E[K^1_{s,t}(x)|K^2]$ a.s.
\end{enumerate}
For $s\leq t$, let
$$\hat {\mathcal F}_{s,t}=\sigma(\hat K_{u,v}, s\leq u\leq v\leq t),\ \ \mathcal F^i_{s,t}=\sigma(K^i_{u,v}, s\leq u\leq v\leq t),\  i=1,2.$$
Then $\hat {\mathcal F}_{s,t}=\mathcal F^1_{s,t}\vee \mathcal F^2_{s,t}$. To simplify notations, we shall assume that $\varphi^c$ is defined on the original space $(\Omega ,\mathcal A,\mathbb P)$ and that (i) and (ii) are satisfied if we replace $(K^1,K^2)$ by $(\delta_{\varphi^c},K)$. Recall that (i) and (ii) are also satisfied  by the pair $(\delta_{\varphi},K^{m^+,m^-})$ constructed in Section \ref{houwa}. Now
\begin{equation}\label{h}
K_{s,t}(x)=E[\delta_{\varphi^{c}_{s,t}(x)}|K]\ \ a.s. \ \textrm{for all}\ \ s\leq t, x\in G,
\end{equation}
and using (\ref{k}), we obtain
\begin{equation}\label{label}
K^{W}_{s,t}(x)=E[\delta_{\varphi^{c}_{s,t}(x)}|\sigma(W)]\ a.s. \ \textrm{for all}\ \ s\leq t, x\in G,
\end{equation} 
 with $K^W$ being the Wiener flow given by (\ref{habla}). 
\begin{prop}
The stochastic flow $\varphi^c$ solves $(E)$.
\end{prop}
\begin{proof}
Fix $t>0, x\in G$. By (\ref{label}), $\delta_{\varphi^c_{0,t}(x)}\ \textrm{is supported on}\ \ \{|Z_{0,t}(x)|\vec{e}_j, 1\leq j\leq N\}\  \ \textrm{a.s.}$ and so $|\varphi^c_{0,t}(x)|=|Z_{0,t}(x)|$. Similarly, using (\ref{label}), we have
\begin{equation}\label{labe}
\varphi^c_{0,t}(x)\in G^+\Leftrightarrow Z_{0,t}(x)\geq0\ \textrm{and}\ \ \varphi^c_{0,t}(x)\in G^-\Leftrightarrow Z_{0,t}(x)\leq0.
\end{equation}
Consequently $\varepsilon(\varphi^{c}_{0,t}(x))=\widetilde{\textrm{sgn}}{(Z_{0,t}(x))}$ a.s. Since $\varphi^{c}_{0,\cdot}(x)$ is an $W(\alpha_1,\cdots,\alpha_N)$ started at $x$, it satisfies Theorem \ref{palestine}; $\forall f\in D(\alpha_1,\cdots,\alpha_N)$,
$$f(\varphi^{c}_{0,t}(x))=f(x)+\int_{0}^{t}f'(\varphi^{c}_{0,u}(x))dB_u+\frac{1}{2}\int_{0}^{t}f''(\varphi^{c}_{0,u}(x))du\ \ a.s.$$
with $B_t=|\varphi_{0,t}(x)|-\tilde L_t(|\varphi_{0,\cdot}(x)|)-|x|=|Z_{0,t}(x)|-\tilde L_t(|Z_{0,\cdot}(x)|)-|x|$. Tanaka's formula and (\ref{labe}) yield $$B_t=\int_{0}^{t}\widetilde{\textrm{sgn}}{(Z_{0,u}(x))}dZ_{0,u}(x)=\int_{0}^{t}\widetilde{\textrm{sgn}}{(Z_{0,u}(x))}W(du)=\int_{0}^{t}\varepsilon(\varphi_{0,u}^{c}(x))W(du).$$
Likewise for all $s\leq t, x\in G, f\in D(\alpha_1,\cdots,\alpha_N)$,
$$f(\varphi^{c}_{s,t}(x))=f(x)+\int_{s}^{t}f'(\varphi^{c}_{s,u}(x))\varepsilon(\varphi^{c}_{s,u}(x))W(du)+\frac{1}{2}\int_{s}^{t}f''(\varphi^{c}_{s,u}(x))du\ \ a.s. $$
\end{proof}
We will see later (Remark \ref{wfa}) that $\varphi^c\overset{law}{=}\varphi$ where $\varphi$ is the stochastic flow of mappings constructed in Section 3.
\subsection{Two probability measures associated to $K$.}\label{cherif}
 For all $t\geq \tau_{s,x}$, set $$V^{+,i}_{s,t}(x)=K_{s,t}(x)({D}_i\setminus\{0\})\ \forall 1\leq i\leq p$$
and
$$ V^{-,N}_{s,t}(x)=K_{s,t}(x)(D_N),\ \ V^{-,i}_{s,t}(x)=K_{s,t}(x)({D}_i\setminus\{0\})\ \ \forall p+1\leq i\leq N-1$$
 $$V^{+}_{s,t}(x)=(V^{+,i}_{s,t}(x))_{1\leq i\leq p}, V^{-}_{s,t}(x)=(V^{-,i}_{s,t}(x))_{p+1\leq i\leq N},\ V_{s,t}(x)=(V^{+}_{s,t}(x),V^{-}_{s,t}(x)).$$
For $s=0$, we use these abbreviated notations
$$Z_t(x)=Z_{0,t}(x),\ \ V^+_t(x)=V^{+}_{0,t}(x),\ \ V^-_t(x)=V^{-}_{0,t}(x),\ V_t(x)=(V^+_t(x),V^-_t(x))$$
and if $x=0$, $$Z_t=Z_{0,t}(0),\ \ V^+_t=V^{+}_{0,t}(0),\ \ V^-_t=V^{-}_{0,t}(0),\ V_t=(V^+_t,V^-_t).$$
\noindent By (\ref{k}), $\forall x\in G,s\leq t$, with probability $1$
\begin{eqnarray}
K_{s,t}(x)&=&\delta_{x+\vec{e}(x)\varepsilon(x)W_{s,t}} 1_{\{t\leq \tau_{s,x}\}}\nonumber\\
&+&(\sum_{i=1}^{p}V^{+,i}_{s,t}(x)\delta_{\vec{e}_i|Z_{s,t}(x)|} 1_{\{Z_{s,t}(x)>0\}}+\sum_{i=p+1}^{N}V^{-,i}_{s,t}(x)\delta_{\vec{e}_i|Z_{s,t}(x)|} 1_{\{Z_{s,t}(x)\leq0\}}) 1_{\{t> \tau_{s,x}\}}.\nonumber\
\end{eqnarray}
Define $$\mathcal F^{K}_{s,t}=\sigma(K_{v,u}, s\leq v\leq u\leq t),\ \ \mathcal F^{W}_{s,t}=\sigma(W_{v,u}, s\leq v\leq u\leq t)$$
 and assume that all these $\sigma$-fields are right-continuous and include all $\mathbb P$-negligible sets. When $s=0$, we denote $\mathcal F^{K}_{0,t}, \mathcal F^{W}_{0,t}$ simply by $\mathcal F^{K}_{t}, \mathcal F^{W}_{t}$. Recall that for all $s\in{\R}, x\in G$, the mapping $t\longmapsto K_{s,t}(x)$ defined from $[s,+\infty[$ into $\mathcal P(G)$ is continuous. Then the following Markov property holds.
\begin{lemma}\label{r}
Let $x, y\in G$ and $T$ be an $(\mathcal F^{K}_{t})_{t\geq 0}$ stopping time such that $K_{0,T}(x)=\delta_y$ a.s. Then $K_{0,\cdot+T}(x)$ is independent of $\mathcal F^{K}_{T}$ and has the same law as $K_{0,\cdot}(y)$.
\end{lemma}
\noindent As a consequence of the preceding lemma, for each $x\in G$, $K_{0,\cdot+\tau_{0,x}}(x)$ is independent of  $\mathcal{F}^{K}_{\tau_{0,x}}$ and is equal in law to $K_{0,\cdot}(0)$.\\
 Consider the following random times:
$$T=\inf\{r\geq 0 : Z_{r}=1\},\ \ L=\sup\{r\in[0,T] : Z_{r}=0\}$$
and the following $\sigma$-fields:
$$\mathcal{F}_{L-}=\sigma(X_L, X\ \textrm{is bounded}\ \  (\mathcal{F}^{W}_{t})_{t\geq 0}-\textrm{previsible process}),$$
$$\mathcal{F}_{L+}=\sigma(X_L, X\ \textrm{is bounded}\ \  (\mathcal{F}^{W}_{t})_{t\geq 0}-\textrm{progressive process}).$$
Then $\mathcal{F}_{L+}=\mathcal{F}_{L-}$ (Lemma 4.11 in \cite{MR2235172}). Let $f : \R^N\longrightarrow\R$ be a bounded continuous function and set $X_t=E[f(V_t)|\sigma(W)]$. Thanks to (\ref{h}), the process $r\longmapsto V_{r}$ is constant on the excursions of $r\longmapsto Z_{r}$. By following the same steps as in Section 4.2 \cite{MR2235172}, we show that there is an $\mathcal F^W$-progressive version of $X$ that is constant on the excursions of $Z$ out of $0$ (Lemma 4.12 \cite{MR2235172}). We take for $X$ this version. Then $X_{T}$ is $\mathcal{F}_{L+}$ measurable and $E[X_T|\mathcal{F}_{L-}]=E[f(V_T)]$ (Lemma 4.13 \cite{MR2235172}). This implies that $V_T$ is independent of $\sigma(W)$ (Lemma 4.14 \cite{MR2235172}) and the same holds if we replace $T$ by $\inf\{t\geq 0 : Z_{t}=a\}$ where $a>0$.\\ Define by induction $T_{0,n}^+=0$ and for $k\geq 1$:
$$S_{k,n}^{+}=\inf\{t\geq T_{k-1,n}^{+} : Z_{t}=2^{-n}\},\ \ T_{k,n}^{+}=\inf\{t\geq S_{k,n}^{+} : Z_{t}=0\}.$$
Set $V_{k,n}^{+}=V_{S_{k,n}^{+}}^{+}$. Then, we have the following 
\begin{lemma}
For all $n,(V_{k,n}^{+})_{k\geq 1}$ is a sequence of i.i.d. random variables. Moreover, this sequence is independent of $W$.
\end{lemma}
\begin{proof}
For all $k\geq2$, $V_{k,n}^{+}$ is $\sigma(K_{0,T_{k-1,n}^{+}+t}(0), t\geq 0)$ measurable
and $V_{k-1,n}^{+}$ is $\mathcal F^{K}_{T_{k-1,n}^{+}}$ measurable which proves the first claim by Lemma \ref{r}. Now, we show by induction on $q$ that $(V_{1,n}^{+},\cdots,V_{q,n}^{+})$ is independent of $\sigma(W)$. For $q=1$, this has been justified. Suppose $(V_{1,n}^{+},\cdots,V_{q-1,n}^{+})$ is independent of $\sigma(W)$ and write $$\sigma(W_{0,u}, u\geq0)=\sigma(Z_{u\wedge T_{q-1,n}^{+}}, u\geq0)\vee \sigma(Z_{u+ T_{q-1,n}^{+}}, u>0).$$
Since $(V_{1,n}^{+},\cdots,V_{q-1,n}^{+})$ is $\mathcal F^{K}_{T_{q-1,n}^{+}}$ measurable and $$\sigma(Z_{u+ T_{q-1,n}^{+}}, u>0)\vee\sigma(V_{q,n}^{+})\subset \sigma(K_{0,T_{q-1,n}^++t}(0), t\geq 0),$$ we conclude that $(V_{1,n}^{+},\cdots,V_{q,n}^{+})$  and $\sigma(W)$ are independent.
\end{proof}
Let $m_n^{+}$ be the common law of $(V_{k,n}^{+})_{k\geq 1}$ for each $n\geq1$ and define $m^+$ as the law of $V_1^+$ under $\mathbb P(.|Z_1>0)$. Then, we have the
\begin{lemma}\label{qods}
The sequence $(m_n^{+})_{n\geq1}$ converges weakly towards $m^+$. For all $t>0$, under $\mathbb P(\cdot|Z_t>0)$, $V^+_t$ and $W$ are independent and the law of $V^+_t$ is given by $m^+$.
\end{lemma}
\noindent\begin{proof}
For each bounded continuous function $f : \R^p\longrightarrow\R,$
$$\begin{array}{ll}
E[f(V^+_{t})|W] 1_{\{Z_{t}>0\}}&=\displaystyle{\lim_{n\rightarrow \,\infty}}\displaystyle{\sum_{k}}
E\big[1_{\{t\in[S_{k,n}^{+},T_{k,n}^{+}[\}}f(V_{k,n}^{+})|W\big]\\
&=\displaystyle{\lim_{n\rightarrow \,\infty}}\displaystyle{\sum_{k}} 1_{\{t\in[S_{k,n}^{+},T_{k,n}^{+}[\}}\left(\int f dm_{n}^+\right)\\
&=[1_{\{Z_t>0\}}\displaystyle{\lim_{n\rightarrow \,\infty}}\int f dm_{n}^++\varepsilon_n(t)]\\
\end{array}$$
with $\displaystyle{\lim_{n\rightarrow \,\infty}}\varepsilon_n(t)=0$ a.s. Consequently $$\displaystyle{\lim_{n\rightarrow \,\infty}}\int f dm_{n}^+=\frac{1}{\mathbb P(Z_t>0)}E[f(V_t^+) 1_{\{Z_t>0\}}].$$
The left-hand side does not depend on $t$, which completes the proof.
\end{proof}
\noindent We define analogously the measure $m^{-}$ by considering the following stopping times: $T_{0,n}^-=0$ and for $k\geq 1$:
$$S_{k,n}^{-}=\inf\{t\geq T_{k-1,n}^{-} : Z_{t}=-2^{-n}\},\ \ T_{k,n}^{-}=\inf\{t\geq S_{k,n}^{-} : Z_{t}=0\}.$$
Set $V_{k,n}^{-}=V_{S_{k,n}^{-}}^{-}$ and let  $m_n^{-}$ be the common law of $(V_{k,n}^{-})_{k\geq 1}$. Denote by $m^-$ the law of $V_1^-$ under $\mathbb P(.|Z_1<0)$. Then, the sequence $(m_n^{-})_{n\geq1}$ converges weakly towards $m^-$. Moreover, for all $t>0$, the law of $V^-_t$ under $\mathbb P(.|Z_t<0)$ is given by $m^-$. As a result, we have
$$E[f(V^-_{t})|W] 1_{\{Z_{t}<0\}}=1_{\{Z_{t}<0\}}\int f dm^{-}$$
for each measurable bounded $f : \R^{N-p}\longrightarrow\R$. If we follow the same steps as before but consider $(Z_{u+\tau_{0,x}}(x), u\geq0)$ for all $x$, we show that the law of $V_{0,t}^+(x)$ under $\mathbb P(.|Z_{0,t}(x)>0, t>\tau_{0,x})$ does not depend on $t>0$. Denote by $m_x^+$ such a law. Then, thanks to Lemma \ref{r}, $m_x^+$ does not depend on $x\in G$. Thus $m_x^+=m^+$ for all $x$ and 
\begin{equation}\label{ramond}
E[f(V^+_{t}(x))|W] 1_{\{Z_{t}(x)>0, t>\tau_{0,x}\}}=1_{\{Z_{t}(x)>0, t>\tau_{0,x} \}}\int f dm^{+}
\end{equation}
for each measurable bounded $f : \R^p\longrightarrow\R$. Similarly
\begin{equation}\label{romond} 
E[h(V^-_{t}(x))|W] 1_{\{Z_{t}(x)<0, t>\tau_{0,x}\}}=1_{\{Z_{t}(x)<0, t>\tau_{0,x} \}}\int h  dm^{-}
\end{equation}
for each measurable bounded $h : \R^{N-p}\longrightarrow\R$.
\subsection{Unicity in law of $K$.}
Define $$p(x)=|x|\vec{e}_1 1_{\{x\in G^+\}}+|x|\vec{e}_{p+1}1_{\{x\in G^-, x\neq0\}},\ \ x\in G.$$
Fix $x\in G$, $0<s<t$ and let $x_s=p(\varphi^c_{0,s}(x))$. Then:
\begin{enumerate}
 \item [(i)] $\varphi^c_{s,r}(x)=x+\vec{e}(x)\varepsilon(x)W_{s,r}$ for all $r\leq \tau_{s,x}$ (from Lemma \ref{z}).
 \item [(ii)] $\tau_{s,x}=\tau_{s,p(x)}$ and $\varphi^c_{s,r}(x)=\varphi^c_{s,r}(p(x))$ for all $r\geq \tau_{s,x}$ since $\varphi^c$ is a coalescing flow.
 \item [(iii)] $\tau_{s,\varphi^c_{0,s}(x)}=\tau_{s,x_s}$ and $\varphi^c_{s,r}(\varphi^c_{0,s}(x))=\varphi^c_{s,r}(x_s)$ for all $r\geq \tau_{s,x_s}$ by (ii) and the independence of increments of $\varphi^c$.
 \item [(iv)] On $\{t>\tau_{s,x_s}\}$, $\varphi^c_{0,t}(x)=\varphi^c_{s,t}(\varphi^c_{0,s}(x))=\varphi^c_{s,t}(x_s)$ by the flow property of $\varphi^c$ and (iii).
 \item [(v)] Clearly $\tau_{s,x_s}=\inf\{r\geq s,\ Z_{0,r}(x)=0\}$ a.s. Since $\{\tau_{0,x}<s<g_{0,t}(x)\}\subset\{t>\tau_{s,x_s}\}$ a.s., we deduce that
$$\mathbb P(\varphi^c_{0,t}(x)=\varphi^c_{s,t}(x_s)|\tau_{0,x}<s<g_{0,t}(x))=1.$$ 
\item [(vi)] Recall that $\hat {\mathcal F}_{0,s}$ and $\hat {\mathcal F}_{s,t}$ are independent ($\hat K$ is a flow) and $\hat {\mathcal F}_{0,t}=\hat {\mathcal F}_{0,s}\vee \hat {\mathcal F}_{s,t}$. By (\ref{h}), we have $K_{s,t}(x_s)=E[\delta_{\varphi^c_{s,t}(x_s)}|\mathcal F^K_{0,t}]$ and as a result of (v),
\begin{equation}\label{enfin}
\mathbb P(K_{s,t}(x_s)=K_{0,t}(x)|\tau_{0,x}<s<g_{0,t}(x))=1. 
\end{equation}
\end{enumerate}
\begin{lemma}\label{nouv}
Let $\mathbb P_{t,x_1,\cdots,x_n}$ be the law of $(K_{0,t}(x_1),\cdots,K_{0,t}(x_{n}),W)$ where $t\geq0$ and $x_1,\cdots,x_n\in G$. Then, 
$\mathbb P_{t,x_1,\cdots,x_n}$ is uniquely determined by $\{\mathbb P_{u,x}, u\geq0, x\in G\}$.
\end{lemma}
\begin{proof}
 We will prove the lemma by induction on $n$. For $n=1$, this is clear. Notice that if $t<\tau_{0,z}$, then  $K_{0,t}(z)$ is $\sigma(W)$ measurable and if $t>T^{z_1,z_2}_{0,0}$, then $K_{0,t}(z_1)=K_{0,t}(z_2)$. Suppose the result holds for $n\geq 1$ and let $x_{n+1}\in G$.  Then by the previous remark, we only need to check that the law of $(K_{0,t}(x_1),\cdots,K_{0,t}(x_{n+1}),W)$ conditionally to $A=\{\displaystyle\sup_{1\leq i\leq n+1}\tau_{0,x_i}<t<T_{0,\cdots,0}^{x_1,\cdots,x_{n+1}}\}$ only depends on $\{\mathbb P_{u,x}, u\geq0, x\in G\}$. Remark that on $A$, $\{g_{0,t}(x_i),1\leq i\leq n+1\}$ are distinct and so by summing over all possible cases, we may replace $A$ by $$E=\{\displaystyle\sup_{1\leq i\leq n+1}\tau_{0,x_i}<t<T_{0,\cdots,0}^{x_1,\cdots,x_{n+1}}, g_{0,t}(x_1)<\cdots<g_{0,t}(x_n)<g_{0,t}(x_{n+1})\}$$ Recall the definition of $f$ from Section \ref{humain} and let $S=f(g_{0,t}(x_n),g_{0,t}(x_{n+1}))$, $E_s=E\cap\{S=s\}$ for $s\in \mathbb D$. Then it will be sufficient to show that the law of $(K_{0,t}(x_1),\cdots,K_{0,t}(x_{n+1}),W)$ conditionally to $E_s$ only depends on $\{\mathbb P_{u,x}, u\geq0, x\in G\}$ where $s\in \mathbb D$ is fixed such that $s<t$. On $E_s$,
\begin{enumerate}
 \item [(i)] $(K_{0,t}(x_1),\cdots,K_{0,t}(x_n),W)$ is a measurable function of $(V_{s}(x_1),\cdots,V_{s}(x_n),W)$ as $(V_r(x_i), r\geq\tau_{0,x_i})$ is constant on the excursions of $(Z_r(x_i),r\geq\tau_{0,x_i})$.
 \item [(ii)] There exists a random variable $X_{n+1}$ which is $\mathcal F_{0,s}^{W}$ measurable and satisfies $K_{0,t}(x_{n+1})=K_{s,t}(X_{n+1})$ (from (\ref{enfin})).
\end{enumerate}
Clearly, the law of $(V_{s}(x_1),\cdots,V_{s}(x_n),K_{s,t}(X_{n+1}),W)$ is uniquely determined by $\{\mathbb P_{s,x_1,\cdots,x_n}, \mathbb P_{t-s,y}, y\in G\}$. This completes the proof.
\end{proof}
\begin{prop}
Let $(K^{m^+,m^-},W')$ be the solution constructed in Section 3 associated with $(m^+,m^-)$. Then $K\overset{law}{=}K^{m^+,m^-}$.
\end{prop}
\begin{proof}
From (\ref{ramond}) and (\ref{romond}), $(K_{0,t}(x),W)\overset{law}{=}(K^{m^+,m^-}_{0,t}(x),W')$ for all $t>0$ and $x\in G$. Notice that all the properties (i)-(v) mentioned just above are satisfied by the flow $\varphi$ constructed in Section 3 and consequently $K^{m^+,m^-}$ satisfies also (\ref{enfin}) using the same arguments. By following the same steps as in the proof of Lemma \ref{nouv}, we show by induction on $n$ that
$$(K_{0,t}(x_1),\cdots,K_{0,t}(x_n),W)\overset{law}{=}(K^{m^+,m^-}_{0,t}(x_1),\cdots,K^{m^+,m^-}_{0,t}(x_n),W')$$ 
for all $t>0, x_1,\cdots,x_n\in G$. This proves the proposition.
\end{proof}

\begin{rem}\label{wfa}
When $K$ is a stochastic flow of mappings, then by definition $$(m^+,m^-)=(\displaystyle{\sum_{i=1}^{p}} \frac{\alpha_i}{\alpha^+} \delta_{(0,..,0,1,0,..,0)},\displaystyle{\sum_{i=p+1}^{N}} \frac{\alpha_i}{\alpha^-} \delta_{(0,..,0,1,0,..,0)}).$$ This shows that there is only one flow of mappings solving $(E)$.  
\end{rem}
\subsection{The case $\alpha^+=\frac{1}{2}, N>2$.}
Let $K^W$ be the flow given by (\ref{habla}), where $Z_{s,t}(x)=\varepsilon (x)|x|+W_t-W_s$. It is easy to verify that $K^W$ is a Wiener flow. \textbf{Fix} $s\in\R, x\in G$. Then, by following ideas of Section \ref{humain}, one can construct a real white noise $W$ and a process $(X_{s,t}^x, t\geq s)$ which is an $W(\alpha_1,\cdots,\alpha_N)$ started at $x$ such that 
\begin{itemize}
 \item{(i)} for all $t\geq s, f\in D(\alpha_1,\cdots,\alpha_N)$,
$$f(X_{s,t}^x)=f(x)+\int_s^t(\varepsilon f')(X_{s,u}^x)W(du) + \frac{1}{2}\int_s^tf''(X_{s,u}^x)du\ \ a.s.$$
\item{(ii)} for all $t\geq s$,  $K^W_{s,t}(x)=E[\delta_{X_{s,t}^x}|\sigma(W)]\ \textrm{a.s}.$
\end{itemize}
By conditioning with respect to $\sigma(W)$ in (i), this shows that $K^W$ solves $(E)$. Now, let $(K,W)$ be any other solution of $(E)$ and set $P^n_t=E[K^{\otimes n}_{0,t}]$. From the hypothesis $\alpha^+=\frac{1}{2}$, we see that $h(x)=\varepsilon (x)|x|$ belongs to $D(\alpha_1,\cdots,\alpha_N)$ and by applying $h$ in $(E)$, we get $K_{0,t}h(x)=h(x)+W_t$. Denote by $(X^{x_1},X^{x_2})$ the two-point motion started at $(x_1,x_2)\in G^2$ associated to $P^2$. Since $|X^{x_i}|$ is a reflected Brownian motion started at $|x_i|$ (Theorem \ref{palestine}), we have $E[|X_t^{x_i}|^2]=t+|x_i|^2$. From the preceding observation $E[h(X_t^{x_1})h(X_t^{x_2})]=E[K_{0,t}h(x_1)K_{0,t}h(x_2)]=h(x_1)h(x_2)+t$ and therefore
$$E[(h(X_t^{x_1})-h(X_t^{x_2})-h(x_1)+h(x_2))^2]=0.$$
This shows that $h(X_t^{x_1})-h(X_t^{x_2})=h(x_1)-h(x_2)$. Now we will check by induction on $n$ that $P^n$ does not depend on $K$. For $n=1$, this follows from Proposition \ref{vi}. Suppose the result holds for $n$ and let $(x_1,\cdots,x_{n+1})\in G^{n+1}$ such that $h(x_i)\neq h(x_j), i\neq j$. Let $\tau_{x_i}=\inf\{r\geq0 : X_r^{x_i}=0\}=\inf\{r\geq0 : h(X_r^{x_i})=0\}$ and $(x_i,x_j)\in G^+\times G^-$ such that $h(x_i)<h(x_k), h(x_h)<h(x_j)$ for all $(x_k,x_h)\in G^+\times G^-$ (when $(x_i,x_j)$ does not exist the proof is simpler). Clearly $\tau_{x_k}$ is a function of $X^{x_h}$ for all $h,k\in[1,n+1]$ and so for all measurable bounded $f : G^{n+1}\longrightarrow\R$,

$$f(X_t^{x_1},\cdots,X_t^{x_{n+1}}) 1_{\{t<\tau_{x_i}, \inf_{1\leq k\leq n+1} \tau_{x_k}=\tau_{x_i}\}}\ \textrm{is a function of}\ X^{x_i}$$
and $$f(X_t^{x_1},\cdots,X_t^{x_{n+1}}) 1_{\{t<\tau_{x_j}, \inf_{1\leq k\leq n+1} \tau_{x_k}=\tau_{x_j}\}}\ \textrm{is a function of}\ X^{x_j}.$$
where $t>0$ is fixed. This shows that $E[f(X_t^{x_1},\cdots,X_t^{x_{n+1}}) 1_{\{t<\inf_{1\leq k\leq n+1} \tau_{x_k}\}}]$ only depends on $P^1$. Consider the following stopping times $$S_0=\inf_{1\leq i\leq n+1}\tau_{x_i}, S_{k+1}=\inf\{r\geq S_k : \exists j\in [1,n+1], X_r^{x_j}=0, X_{S_k}^{x_j}\neq0\}, k\geq0.$$ 
Remark that $(S_k)_{k\geq0}$ is a function of $X^{x_h}$ for all $h\in [1,n+1]$. By summing over all possible cases we need only check the unicity in law of $(X_t^{x_1},\cdots,X_t^{x_{n+1}})$ conditionally to $A=\{S_k<t<S_{k+1}, X_{S_k}^{x_h}=0\}$ where $k\geq0, h\in [1,n+1]$ are fixed. Write $A=B\cap\{t-S_k<T\}$ where $B=\{S_k<t, X_{S_k}^{x_h}=0\}=\{S_k<t, X_{S_k}^{x_i}\neq0\ \textrm{if}\  i\neq h\}$ and $T=\inf\{r\geq0, \exists j\neq h : X^{x_j}_{r+S_k}=0\}$. On $A$, $X_t^{x_i}$ is a function of $(X_{S_k}^{x_i},X_t^{x_h})$ and therefore for all measurable bounded $f : G^{n+1}\longrightarrow\R$,
$$f(X_t^{x_1},\cdots,X_t^{x_{n+1}}) 1_A\ \textrm{may be written as}\ \ g((X_{S_k}^{x_i})_{i\neq h},X_t^{x_{h}}) 1_A$$ 
where $g$ is measurable bounded from $G^{n+1}$ into $\R$. By the strong Markov property for $X=(X^{x_1},\cdots,X^{x_{n+1}})$, we have
$$1_{B}E[1_{\{t-S_k<T\}}g((X_{S_k}^{x_i})_{i\neq h},X_t^{x_{h}})|\mathcal F_{S_k}^{X}]=1_{B}\psi(t-S_k,(X_{S_k}^{x_i})_{i\neq h})$$
where
$$\psi(u,y_1,\cdots,y_n)=E[1_{\{u<\inf\{r\geq0 : \exists j\in[1,n], X_r^{y_j}=0\}\}}g(y_1,\cdots,y_n,X_u^{0})].$$
This shows that $E[f(X_t^{x_1},\cdots,X_t^{x_{n+1}})1_A]$ only depends on the law of $(X^{x_i})_{i\neq h}$. As a result, $P_t^{n+1}((x_1,\cdots,x_{n+1}),dy)$ is unique whenever $h(x_i)\neq h(x_j), i\neq j$ and by an approximation argument for all $(x_1,\cdots,x_{n+1})\in G^{n+1}$. Since a stochastic flow of kernels is uniquely determined by the compatible system of its $n$-point motions, this proves $(2)$ of Theorem \ref{hajri}.
\section{Appendix: Freidlin-Sheu formula.}
In this section, we shall prove Theorem \ref{palestine}. We begin by\\
\textbf{Preliminary remarks.}  We recall that if $Y$ is a semimartingale satisfying $\left\langle {Y}\right\rangle=\left\langle {|Y|}\right\rangle$ then $\tilde L_t(Y)=\tilde L_t(|Y|)$. Let $L_t(Y)$ be the (non symmetric) local time at $0$ of $Y$ and $\alpha\in[0, 1]$. If $Y$ is a $SBM(\alpha)$, then $L_{t}(Y)=2\alpha\tilde L_{t}(Y)$ by identifying Tanaka's formulas for symmetric and non symmetric local time for $Y$.\\
Let $Q$ be the semigroup of the reflecting Brownian motion on $\R$ and define $\varPhi(x)=|x|$. Then $X_t=\varPhi(Z_t)$ and it can be easily checked that $P_{t}(f\circ\varPhi)=Q_tf\circ\varPhi$ for all bounded measurable function $f : \R\longrightarrow\R$ which proves (i). (ii) is an easy consequence of Tanaka's formula for local time.\\
(iii) Set $\tau_z=\inf\{r\geq 0 , Z_r=0 \}$. For $t\leq \tau_z$, (\ref{Najia}) holds from It\^o's formula applied to the semimartingale $X$.
By discussing the cases $t\leq \tau_z$ and $t>\tau_z$, one can assume that $z=0$ and so in the sequel we take $z=0$.\\
For all $i\in [1,N]$, define $Z_{t}^i = |Z_t| 1_{\{Z_t\in D_i\}}-|Z_{t}| 1_{\{Z_t\notin D_i\}}$. Then $Z^i_t=\varPhi^i(Z_t)$ where $\varPhi^i(x)=|x| 1_{\{x\in D_i\}}-|x| 1_{\{x\notin D_i\}}$. Let $Q^i$ be the semigroup of the $SBM(\alpha_i)$. Then the following relation is easy to check: $P_{t}(f\circ\varPhi^i)=Q_t^if\circ\varPhi^i$ for all bounded measurable function $f :\R\longrightarrow\R$ which shows that $Z^i$ is a $SBM(\alpha_i)$ started at $0$. We use the notation $(\mathbb P)$ to denote the convergence in probability.

\noindent Let $\delta> 0$. Define $\tau_{0}^\delta=\theta_{0}^\delta=0$ and for $n\geqslant 1$	
$$\theta_{n}^{\delta}=\inf\{r\geq \tau_{n-1}^{\delta} ,|Z_r|=\delta\},\ \ \tau_{n}^{\delta}=\inf\{r\geq \theta_{n}^\delta , Z_r=0\}.$$
Let $f\in C_b^2(G^{*})$ and $t>0$. Then
$$f(Z_t)-f(0)=\sum_{n=0}^{\infty} f(Z_{\theta_{n+1}^\delta\wedge t})-f(Z_{\theta_{n}^\delta\wedge t})=Q_1^{\delta}+Q_2^{\delta}+Q_3^{\delta}$$ 
where\\
$Q_1^{\delta}=\displaystyle{\sum_{n=0}^{\infty}(f(Z_{\theta_{n+1}^\delta\wedge t})-f(Z_{\tau_{n}^\delta\wedge t}))}
-\sum_ {i=1}^{N}\sum_{n=0}^{\infty} \delta f_i'(0+) 1_{\{\theta_{n+1}^\delta\leq t,Z_{\theta_{n+1}^\delta}\in D_i\}},$\\
\\
$Q_2^{\delta}=\displaystyle{\sum_ {i=1}^{N}\sum_{n=0}^{\infty}\delta f_i'(0+) 1_{\{\theta_{n+1}^\delta\leq t,Z_{\theta_{n+1}^\delta}\in D_i\}}},\ \ Q_3^{\delta}=\displaystyle{\sum_{n=0}^{\infty} f(Z_{\tau_{n}^\delta\wedge t})-f(Z_{\theta_{n}^\delta\wedge t})}.$\\
We first show that $Q_1^{\delta}\xrightarrow[\text{$\delta\rightarrow 0 $}]\ 0$ $(\mathbb P)$ and for this write $Q_1^{\delta}=Q_{(1,1)}^{\delta}+Q_{(1,2)}^{\delta}$ with\\
$Q_{(1,1)}^{\delta}=\displaystyle{\sum_{n=0}^{\infty}\sum_{i=1}^{N}(f(Z_{\theta_{n+1}^{\delta}})-f(Z_{\tau_{n}^{\delta}})-\delta f_i'(0+)) 1_{\{\theta_{n+1}^\delta\leq t,Z_{\theta_{n+1}^{\delta}}\in D_i\}}}$,\\
$Q_{(1,2)}^{\delta}=\displaystyle{\sum_{n=0}^{\infty}\sum_{i=1}^{N}(f(Z_{t})-f(Z_{\tau_{n}^\delta\wedge t})) 1_{\{\theta_{n+1}^\delta > t,Z_{\theta_{n+1}^{\delta}}\in D_i\}}}$.\\
Since $f\in C_b^2(G^*)$, we have\\
(i)$\forall i\in [1,N];\ \ \ \displaystyle\int_{0}^{\delta}(f_{i}'(u)-f_{i}'(0+))du=f_{i}(\delta)-f_{i}(0)-\delta f_i'(0+) $.\\
(ii) There exists $M>0$ such that $\forall i\in [1,N],\ u\geq 0 : |f_{i}'(u)-f_{i}'(0+)|\leq Mu$.\\
Consequently\\
$|Q_{(1,1)}^{\delta}|=|\displaystyle{\sum_{n=0}^{\infty}\sum_{i=1}^{N}(f_{i}(\delta)-f_{i}(0)-\delta f_i'(0+)) 1_{\{\theta_{n+1}^\delta\leq t,Z_{\theta_{n+1}^{\delta}}\in D_i\}}}|\leq \frac{NM\delta^2}{2}\displaystyle\sum_ {n=0}^{\infty} 1_{\{\theta_{n+1}^\delta\leq t\}}$.\\
It is known that $\delta\displaystyle\sum_{n=0}^{\infty} 1_{\{\theta_{n+1}^\delta\leq t\}} \xrightarrow[\text{$\delta\rightarrow 0 $}]\ \frac{1}{2}L_{t}(X)$ $(\mathbb P)$ (\cite{MR1725357}) and therefore $Q_{(1,1)}^{\delta}\xrightarrow[\text{$\delta\rightarrow 0 $}]\ 0$ $(\mathbb P)$.\\
Let $C>0$ such that $\forall i\in [1,N],\ u\geq 0 : |f_{i}(u)-f_{i}(0)|\leq Cu$. Then
\begin{eqnarray}
|Q_{(1,2)}^{\delta}|&=&|\sum_{n=0}^{\infty}\sum_{i=1}^{N} (f(Z_{t})-f(Z_{\tau_{n}^\delta\wedge t})) 1_{\{\theta_{n+1}^\delta > t,Z_{\theta_{n+1}^{\delta}}\in D_i\}}|\nonumber\\
&\leqslant&\sum_{n=0}^{\infty}\sum_{i=1}^{N} |f_{i}(X_{t})-f_{i}(0)| 1_{\{\tau_{n}^\delta< t< \theta_{n+1}^\delta ,Z_{\theta_{n+1}^{\delta}}\in D_i\}}\nonumber\\
&\leqslant&CX_t\sum_{n=0}^{\infty} 1_{\{\tau_{n}^\delta< t< \theta_{n+1}^\delta\}}\leq {C\delta}\nonumber\
\end{eqnarray}
which shows that \ $Q_{(1,2)}^{\delta}\xrightarrow[\text{$\delta\rightarrow 0 $}]\ 0$\ \ a.s. and so $Q_{1}^{\delta}\xrightarrow[\text{$\delta\rightarrow 0 $}]\ 0 \ $ $(\mathbb P)$.\\
Now define $Q_{(2,i)}^{\delta}=\delta\displaystyle{\sum_{n=0}^{\infty} 1_{\{\theta_{n+1}^\delta\leq t,Z_{\theta_{n+1}^\delta}\in D_i\}}}$. Since $\displaystyle{\sum_{n=0}^{\infty} 1_{\{\theta_{n+1}^\delta\leq t,Z_{\theta_{n+1}^\delta}\in D_i\}}}$ is the number of upcrossings of $Z^i$ from $0$ to $\delta$ before time $t$, we have $Q_{(2,i)}^{\delta}\xrightarrow[\text{$\delta\rightarrow 0 $}]\ \frac{1}{2}L_{t}(Z^{i})$ $(\mathbb P)$. Using our preliminary remarks, we see that $Q_{2}^{\delta}\xrightarrow[\text{$\delta\rightarrow 0 $}]\ (\displaystyle{\sum_{i=1}^{N}\alpha_if_{i}'(0+))\tilde L_t(X)}\ (\mathbb P)$.\\
We now establish that $Q_{3}^{\delta}\xrightarrow[\text{$\delta\rightarrow 0 $}]\ {\displaystyle\int_{0}^{t}f'(Z_s)dB_s+\frac{1}{2}\displaystyle\int_{0}^{t}f''(Z_s)ds}\ $ $(\mathbb P)$. For this write $Q_{3}^{\delta}=Q_{(3,1)}^{\delta}+Q_{(3,2)}^{\delta}$ with
$$Q_{(3,1)}^{\delta}=\sum_{n=0}^{\infty}(f(Z_{\tau_{n}^{\delta}})-f(Z_{\theta_{n}^{\delta}})) 1_{\{\tau_{n}^\delta\leq t\}}=\sum_{n=0}^{\infty}\sum_{i=1}^{N}(f(0)-f_{i}(\delta)) 1_{\{\tau_{n}^\delta\leq t,Z_{\theta_{n}^\delta}\in D_i\}},$$
$$Q_{(3,2)}^{\delta}=\sum_{i=1}^{N}(f(Z_t)-f_{i}(\delta))\sharp\{n\in\N : \theta_{n}^\delta< t<\tau_{n}^\delta, Z_{\theta_{n}^\delta}\in D_i\}.$$
It is clear that $\sharp\{n\in\N : \theta_{n}^\delta< t<\tau_{n}^\delta, Z_{\theta_{n}^\delta}\in D_i\}\xrightarrow[\text{$\delta\rightarrow 0 $}]\ 1_{\{Z_{t}\in {D_i\setminus\{0\}}\}}$ a.s. and so $Q_{(3,2)}^{\delta}$ converges to $f(Z_t)-f(0)$ as $\delta\rightarrow0$ a.s. Define $\tau_{0}^{\delta,i}=\theta_{0}^{\delta,i}=0$ and 	
$$\theta_{n}^{\delta,i}=\inf\{r\geq \tau_{n-1}^{\delta,i} ,Z_r=\delta\vec{e}_i\};\ \ \tau_{n}^{\delta,i}=\inf\{r\geq \theta_{n}^{\delta,i} , Z_r=0\}, n\geqslant 1.$$
Using ${\sum_{n=0}^{\infty} 1_{\{\tau_{n}^\delta\leq t,Z_{\theta_{n}^\delta}\in D_i\}}}=\sum_{n=0}^{\infty} 1_{\{\tau_{n}^{\delta,i}\leq t\}}$, we get $Q_{(3,1)}^{\delta}=\sum_{n=0}^{\infty}\sum_{i=1}^{N}(f(0)-f_{i}(\delta)) 1_{\{\tau_{n}^{\delta,i}\leq t\}}.$
\\
On the other hand
$$\sum_{n=0}^{\infty} (f_{i}(X_{\tau_{n}^{\delta,i}\wedge t})-f_{i}(X_{\theta_{n}^{\delta,i}\wedge t}))=\sum_{n=0}^{\infty} (f_{i}(X_{\tau_{n}^{\delta,i}})-f_{i}(X_{\theta_{n}^{\delta,i}})) 1_{\{\tau_{n}^{\delta,i}\leq t\}}+\sum_{n=0}^{\infty} (f_{i}(X_{t})-f_{i}(0)) 1_{\{\theta_{n}^{\delta,i}< t<\tau_{n}^{\delta,i}\}}$$
and therefore $$Q_{(3,1)}^{\delta}=\sum_{i=1}^{N}\sum_ {n=0}^{\infty} (f_{i}(X_{\tau_{n}^{\delta,i}\wedge t})-f_{i}(X_{\theta_{n}^{\delta,i}\wedge t}))-\sum_{i=1}^{N} (f_{i}(X_{t})-f_{i}(0))\times \sharp \{n\in{\N},\theta_{n}^{\delta,i}<t<\tau_{n}^{\delta,i} \}.$$
Since $\ \sharp \{ n\in {\N},\theta_{n}^{\delta,i}<t<\tau_{n}^{\delta,i} \}\xrightarrow[\text{$\delta\rightarrow 0 $}]\ 1_{\{Z_{t}\in {D_i\setminus\{0\}}\}}$ a.s., we deduce that $$Q_{3}^{\delta}\overset{\delta\rightarrow 0}{=}\sum_{i=1}^{N}\sum_ {n=0}^{\infty} (f_{i}(X_{\tau_{n}^{\delta,i}\wedge t})-f_{i}(X_{\theta_{n}^{\delta,i}\wedge t}))+o(1)\ \  a.s.$$
For all $i\in [1,N]$, let $\tilde f_i$ be $C^2$ on $\R$ such that $\tilde f_i=f_i$ on ${\R_+},$\ \ \ $\tilde f_i^{'}=f_i^{'},\ \tilde f_i^{''}=f_i^{''}$ on ${\R}_{+}^{*}$.\\
Now a.s.$$\forall s\in [0,t] ,i\in [1,N]\ \ \ \sum_{n=0}^{\infty} 1_{[\theta_{n}^{\delta,i}\wedge t,\tau_{n}^{\delta,i}\wedge t[}(s)\xrightarrow[\text{$\delta\rightarrow 0 $}]\ 1_{\{Z_{s}\in {D_i\setminus\{0\}}\}}.$$
By dominated convergence for stochastic integrals, \\
$\displaystyle{\sum_{n=0}^{\infty} \tilde f_{i}(X_{\tau_{n}^{\delta,i}\wedge t})-\tilde f_{i}(X_{\theta_{n}^{\delta,i}\wedge t})}=\int_{0}^{t}\displaystyle{\sum_{n=0}^{\infty}} 1_{[\theta_{n}^{\delta,i}\wedge t,\tau_{n}^{\delta,i}\wedge t[}(s)d\tilde f_{i}(X_s)\xrightarrow[\text{$\delta\rightarrow 0 $}]\ \int_{0}^{t} 1_{\{Z_{s}\in {D_i\setminus\{0\}}\}}d\tilde f_{i}(X_s)\ \ (\mathbb P).$\\
Finally
\begin{eqnarray}
\int_{0}^{t}f'(Z_s)dB_s+\frac{1}{2}\int_{0}^{t}f''(Z_s)ds&=&\sum_{i=1}^{N}\int_{0}^{t} 1_{\{Z_{s}\in {D_i\setminus\{0\}}\}}(f_{i}'(X_s)dB_s+\frac{1}{2}f_{i}''(X_s)ds)\nonumber\\
&=&\sum_{i=1}^{N}\int_{0}^{t} 1_{\{Z_{s}\in{D_i\setminus\{0\}}\}}d\tilde f_{i}(X_s).\ \ \nonumber\
\end{eqnarray}
by It\^o's formula and using the fact that $d\tilde L_s(X)$ is carried by $\{s : Z_s=0\}$. Now the proof of Theorem \ref{palestine} is complete.
\\ 
\nocite{*}

\textbf{Acknowledgements.} I sincerely thank my supervisor Yves Le Jan for his guidance through my Ph.D. thesis and for his careful reading of this paper. I am also grateful to Olivier Raimond for very useful discussions, for his careful reading of this paper and for his assistance in proving Lemma \ref{sophie}.

\bibliographystyle{plain}
\bibliography{mabibli1}

\end{document}